\documentclass[reqno, english,11pt]{smfart}

\usepackage{amsfonts, amsthm, amsmath, amssymb}
\usepackage{a4wide}

\usepackage{hyperref}
\hypersetup{colorlinks=true,citecolor=green,linkcolor=blue,urlcolor=red,
pdfstartview=FitH}

\RequirePackage{mathrsfs} \let\mathcal\mathscr

\numberwithin{equation}{section}

\newtheorem{theorem}{Theorem}
\newtheorem{lemma}{Lemma}

\theoremstyle{definition}
\newtheorem*{ack}{Acknowledgements}

\newtheorem*{notat}{Notation and conventions}
\newtheorem*{hyp}{Hypothesis-$\rho$}

\renewcommand{\d}{\mathrm{d}}
\renewcommand{\phi}{\varphi}

\newcommand{\0}{\mathbf{0}}

\newcommand{\PP}{\mathbb{P}}
\renewcommand{\AA}{\mathbb{A}}
\newcommand{\FF}{\mathbb{F}}
\newcommand{\ZZ}{\mathbb{Z}}

\newcommand{\NN}{\mathbb{N}}
\newcommand{\QQ}{\mathbb{Q}}
\newcommand{\RR}{\mathbb{R}}
\newcommand{\CC}{\mathbb{C}}

\newcommand{\cQ}{\mathcal{Q}}
\newcommand{\cD}{\mathcal{D}}
\newcommand{\cM}{\mathcal{M}}

\renewcommand{\leq}{\leqslant}
\renewcommand{\le}{\leqslant}
\renewcommand{\geq}{\geqslant}

\renewcommand{\bar}{\overline}

\newcommand{\ma}{\mathbf}

\newcommand{\m}{\mathbf{m}}
\newcommand{\M}{\mathbf{M}}
\newcommand{\A}{\mathbf{A}}
\newcommand{\B}{\mathbf{B}}

\newcommand{\x}{\mathbf{x}}
\newcommand{\y}{\mathbf{y}}

\renewcommand{\v}{\mathbf{v}}
\newcommand{\z}{\mathbf{z}}
\renewcommand{\b}{\mathbf{b}}
\renewcommand{\a}{\mathbf{a}}
\renewcommand{\k}{\mathbf{k}}
\renewcommand{\t}{\mathbf{t}}

\newcommand{\al}{\alpha}

\renewcommand{\rho}{\varrho}

\newcommand{\ve}{\varepsilon}

\newcommand{\bla}{\boldsymbol{\lambda}}

\newcommand{\bxi}{\boldsymbol{\xi}}

\DeclareMathOperator{\rank}{rank}

\DeclareMathOperator{\supp}{supp}

\DeclareMathOperator{\diag}{diag}

\DeclareMathOperator{\Mod}{mod} 
\renewcommand{\bmod}[1]{\,(\Mod{#1})}

\begin{document}

\title[Singular intersections of quadrics]{Rational
  points on singular\\ intersections of quadrics}

\author{T.D.\ Browning}
\address{School of Mathematics\\
University of Bristol\\ Bristol\\ BS8 1TW\\ United Kingdom}
\email{t.d.browning@bristol.ac.uk}

\author{R.\ Munshi}
\address{School of Mathematics\\ 
Tata Institute of Fundamental Research\\
1 Homi Bhabha Road\\ Colaba\\Mumbai 400005\\ India}
\email{rmunshi@math.tifr.res.in}

\date{\today}

\begin{abstract}
Given an intersection of two quadrics $X\subset \PP^{m-1}$, with
$m\geq 9$, the quantitative arithmetic of the set $X(\QQ)$ is
investigated under the assumption that the singular locus of $X$ 
consists of a pair of conjugate singular points defined over $\QQ(i)$.  
\end{abstract}


\maketitle
\tableofcontents

\section{Introduction}
\label{intro}

The arithmetic of quadratic forms has long held a special place in
number theory. In this paper we focus our efforts on 
algebraic varieties $X\subset \PP^{m-1}$ which arise as the
common zero locus of two quadratic forms $q_1,q_2\in 
\ZZ[x_1,\dots,x_m]$.  We will always assume that $X$ is a 
geometrically integral complete intersection which is not a cone. 
Under suitable further hypotheses on $q_1$ and 
$q_2$, we will be concerned with estimating 
 the number of $\QQ$-rational points on $X$ of bounded height. 
Where successful this will be seen to yield a proof of the Hasse
principle for the varieties under consideration. 

The work of Colliot-Th\'el\`ene, Sansuc and Swinnerton-Dyer \cite{CT} 
provides a comprehensive description of the qualitative arithmetic 
associated to the set $X(\QQ)$ of $\QQ$-rational points on $X$
for large enough values of $m$. In
fact it is known that the Hasse principle holds for any smooth model
of $X$ if $m\geq 9$. This can be
reduced to $m\geq 5$ provided that  $X$
contains a pair of conjugate singular points and does not belong to a certain explicit class of varieties for which the Hasse principle is known to fail.

In this paper the quadratic forms $q_1$ and $q_2$ will have special
structures. Let $Q_1$ and $Q_2$ be  integral
quadratic forms in $n$ variables $\mathbf x=(x_1,\dots,x_n)$, 
with  underlying symmetric matrices $\M_1$ and $\M_2$, so that
$Q_i(\x)=\x^T\M_i\x$ for $i=1,2$. Then we set 
\begin{align*}
q_1(x_1,\dots, x_{n+2})&=Q_1(x_1,\dots,x_n)-x_{n+1}^2-x_{n+2}^2,\\
q_2(x_1,\dots, x_{n+2})&=Q_2(x_1,\dots,x_n).
\end{align*}
We will henceforth assume that $Q_{2}$ is non-singular and that as a
variety $V$ in $\PP^{n-1}$, the intersection of quadrics
$Q_{1}(\x)=Q_{2}(\x)=0$ is also non-singular. 
It then follows that $X$ has a singular locus 
containing precisely two singular points which are conjugate over
$\QQ(i)$. The question of whether the Hasse principle holds for 
such varieties is therefore answered in the affirmative by  \cite{CT} when $n\geq
3$.   Furthermore, when $X(\QQ)$ is non-empty, it is well-known  
(see \cite[Proposition 2.3]{CT}, for example)  that $X$ is $\QQ$-unirational. 
In particular $X(\QQ)$ is Zariski dense in $X$ as soon as it is non-empty.

Let $r(M)$ be the function that counts the number of representations
of an integer $M$ as a sum of two squares and  
let $W: \RR^{n}\rightarrow \RR_{\geq 0}$ be an infinitely
differentiable bounded function of compact support.
Our analysis of the density of $\QQ$-rational points on $X$ will be
activated via  the weighted sum 
\begin{equation}\label{eq:main-sum}
S(B)=\sum_{\substack{\mathbf x \in \mathbb Z^n\\ 2\nmid Q_{1}(\x)\\
Q_2(\x)=0}}r(Q_1(\mathbf
x)) W\left(\frac{\mathbf x}{B}\right), 
\end{equation}
for $B\rightarrow \infty$.  The requirement that  $Q_{1}(\x)$ be odd 
is not strictly  necessary but makes our argument technically simpler.
Simple heuristics lead one to expect that $S(B)$  has order of magnitude $B^{n-2}$,
provided that there are points in $X(\RR)$ and $X(\QQ_{p})$ for every
prime $p$.   Confirmation of this fact is provided by 
work of Birch \cite{birch} when $n\geq 12$.   
 Alternatively, when
$Q_{1}$ and $Q_{2}$ are both diagonal and the form
$b_{1}q_{1}+b_{2}q_{2}$ is indefinite and has rank at least $5$
for every non-zero pair  $(b_{1},b_{2})\in \RR^2$, then Cook
\cite{C} shows that $n\geq 7$ is permissible.  
The following result offers an improvement over both of these results.

\begin{theorem}
\label{th1}
Let $n\geq 7$ and assume that 
$V$ is  non-singular with  $Q_{2}$ also non-singular.
Assume that  
$Q_1(\x)\gg 1$ and $\nabla Q_1(\x)\gg 1$,
for some absolute implied constant, for every $\x \in
\supp(W)$.
Suppose that $X(\RR)$ and $X(\QQ_{p})$ are non-empty for each prime $p$.  
Then there exist constants $c>0$ and $\delta>0$ such that 
$$
S(B)=cB^{n-2}+O(B^{n-2-\delta}).
$$
The  implied constant is allowed to
depend on 
$Q_1, Q_{2}$ and $W$.
\end{theorem}

In 
\S \ref{s:conclusion}
an explicit value of $\delta$ will be given and it will be 
explained that  the
leading constant is an absolutely convergent 
product of local densities
$
c=\sigma_\infty \prod_p \sigma_p,
$
whose positivity is equivalent to the 
hypothesis that  $X(\RR)$ and $X(\QQ_{p})$ are non-empty for each prime $p$.  
In particular Theorem \ref{th1} provides a new proof of the Hasse principle for the varieties $X$ under consideration. 

Our proof of Theorem \ref{th1}  uses the circle method.  An
inherent technical difficulty in applying the circle method to systems
of more than one equation lies in the lack of a suitable analogue of
the Farey dissection of the unit interval, as required for the
so-called ``Kloosterman refinement''. In the present case this
difficulty is circumvented by the specific shape of the quadratic
forms $q_{1},q_{2}$. Thus it is possible to trade the equality
$Q_{1}(\x)=x_{n+1}^{2}+x_{n+2}^{2}$ for a family of congruences using the
familiar identity  
$$
r(M)=4\sum_{d\mid M}\chi(d),
$$
where $\chi$ is the real non-principal character modulo $4$. 
In this fashion the sum $S(B)$ can be thought of as counting suitably weighted
 solutions  $\x\in \ZZ^{n}$ of the quadratic equation $Q_{2}(\x)=0$,
 for which $Q_{1}(\x)\equiv 0 \bmod{d}$, for  varying $d$.
 We will apply the circle method to detect the single equation $Q_{2}(\x)=0$, 
 in the form developed by Heath-Brown \cite{H},  thereby setting the
 scene for a double  Kloosterman refinement by way of  Poisson summation.  
 This approach ought to be
 compared with joint work of the second author with Iwaniec 
\cite{IM2}, wherein an upper bound is achieved for the number of
integer solutions in a box to the pair of quadratic equations  
$Q_1(\mathbf x)=\Box$ and $Q_2(\mathbf x)=0$,
when $n=4$. In this case a simple upper bound sieve is
used to detect the square,  which thereby allows  the first equation to
be exchanged for a suitable family of congruences.   
Finally we remark that with additional work it would be possible to work with more general quadrics, in which the term $x_{n+1}^2+x_{n+2}^2$ is replaced by an arbitrary positive definite binary quadratic form. 

The exponential sums that feature in our work take the shape 
\begin{align}\label{eq:S'}
S_{d,q}(\m)=\sideset{}{^{*}}\sum_{a\bmod{q}}
\sum_{\substack{\mathbf k \bmod{dq}\\
Q_1(\k)\equiv 0\bmod{d}\\Q_2(\k)\equiv 0\bmod{d}}}
e_{dq}\left(aQ_2(\k)+\m.\k\right),
\end{align}
for positive integers $d$ and $q$ and varying $\m\in \ZZ^{n}$. 
The notation $\sum^*$ means that the sum is taken over elements coprime to the modulus.
We will extend it to summations over vectors in the obvious way.
There is a basic 
multiplicativity relation at work which renders it profitable 
 to consider the cases $d=1$ and $q=1$ separately. In the former case
 we will need to gain sufficient cancellation in the sums that emerge
 by investigating the analytic properties of the associated Dirichlet
 series 
$$
\xi(s;\m)=\sum_{q=1}^{\infty} \frac{S_{1,q}(\m)}{q^s},
$$ 
for $s\in \CC$.  
This is facilitated by the
fact that 
$S_{1,q}(\m)$  can be evaluated
explicitly using the formulae for quadratic Gauss sums. We will see
in \S \ref{sec:qsum}
that $\xi(s;\m)$ is absolutely convergent for
$\Re(s)>\frac{n}{2}+2$.  
In order to prove Theorem \ref{th1} it is important to establish an
analytic continuation of $\xi(s;\m)$ to the left of this
line.   This eventually allows us to establish an 
asymptotic formula for $S(B)$ provided that $n>6$. 
The situation for $n=6$ is more delicate and we are no longer able to
win sufficient cancellation through an analysis of $\xi(s;\m)$
alone. In fact it appears desirable to exploit  cancellation due to
sign changes in the exponential sum $S_{d,1}(\m)$. The latter 
is associated to a pair of quadratic forms, rather than a single form,
and this raises significant technical obstacles. We 
intend to return to this topic in a future publication.

With a view to subsequent refinements, much of our argument works under much greater generality than 
for the quadratic forms considered in Theorem \ref{th1}. 
In line with this, unless otherwise indicated, any estimate concerning  quadratic forms 
$Q_1,Q_2\in \ZZ[x_1,\ldots,x_n]$ 
is valid for arbitrary forms such that  $Q_2$ is non-singular, $n\geq 4$ and the variety $V\subset \PP^{n-1}$ 
defined by  $Q_1(\x)=Q_2(\x)=0$ 
is a 
geometrically integral complete intersection.
We let 
$$
\rho(d)=S_{d,1}(\mathbf{0}),
$$
in the notation of \eqref{eq:S'}.
The Lang--Weil estimate yields
$\rho(p)=O(p^{n-2})$
when $d=p$ is a prime, since the affine cone over $V$ has dimension  $n-2$.  
We will need upper bounds  for $\rho(d)$
of comparable strength for any $d$.  It will be convenient to make the following hypothesis.

\begin{hyp}
Let  $d\in \NN$ and $\ve>0$. Then we have 
 $\rho(d)=O(d^{n-2+\ve})$.
\end{hyp}

Here, as throughout our work,  the implied constant is allowed to
depend upon the coefficients of the quadratic forms $Q_{1},Q_{2}$ under
consideration and the
parameter $\ve$.  We will further allow all our implied constants to depend on the weight function $W$ in \eqref{eq:main-sum}, with any further dependence being explicitly
indicated by appropriate subscripts.
We will establish Hypothesis-$\rho$ in Lemma \ref{rho(d)} when $V$ is non-singular, as required for Theorem \ref{th1}.

\begin{notat}
Throughout our work $\NN$ will denote the set of positive
integers.  The
parameter $\ve$ will always denote a small positive real
number, which is allowed to take different values at different parts
of the argument.  We shall use $|\x|$ to denote the norm $\max |x_i|$ 
of a vector $\x=(x_1,\dots,x_n)\in \RR^{n}$. 
Next, given integers $m$ and $M$, by writing 
$m\mid M^\infty$ we will mean that any prime divisor of $m$ is also a prime
divisor of $M$.   Likewise $(m,M^\infty)$ is taken to mean the largest positive 
divisor $h$ of $m$ for which $h\mid M^\infty$.
It will be convenient to record the bound
\begin{equation}\label{eq:scat}
\#\{m\leq x: m\mid M^\infty\} 
\leq \sum_{p\mid m \Rightarrow p\mid M} \left(\frac{x}{m}\right)^\ve =
x^\ve \prod_{p\mid M}\left(1-p^{-\ve}\right)^{-1}
\ll (x|M|)^\ve,
\end{equation}
for any $x\geq 1$, 
a fact that we shall make frequent use of in our work. 
Finally we will write 
$e(x)=\exp(2\pi ix)$ and 
$e_q(x)=\exp(\frac{2\pi ix}{q})$.
\end{notat}

\begin{ack}
Some of this work was done while the authors were both visiting the 
 {\em Institute
for Advanced Study} in Princeton, 
the hospitality and financial
support of which is gratefully acknowledged. 
While working on this paper the first author was 
supported by EPSRC grant number
\texttt{EP/E053262/1}. The authors are very grateful to the anonymous referee for 
numerous helpful comments and for drawing our attention to  an error in the original treatment of Lemma \ref{lem:technical}.
\end{ack}

\section{Auxiliary estimates}

\subsection{Linear congruences}
\label{s:congruences}

Let $q\in \NN$.
For $n\times n$ matrices $\M$, with coefficients in $\ZZ$, and a vector $\a\in\ZZ^n$ we will often be led to
consider the cardinality
\begin{equation}\label{eq:2.1}
K_{q}(\M;\a)=\#\{\x\bmod{q}: \M\x\equiv \mathbf{a} \bmod{q}\}.
\end{equation}
The Chinese remainder theorem implies that $K_q(\M;\a)$ is a multiplicative function
of $q$, rendering it sufficient to conduct our analysis at prime powers  $q=p^r$.
We will need the following basic upper bound.

\begin{lemma}\label{lem:smith}
Assume that $\M$ has rank $\rho$ and let $\delta_p$ be the  
minimum of the $p$-adic orders of the $\rho \times \rho$ non-singular
submatrices of $\M$. Then we have 
$$
K_{p^r}(\M;\a)\leq \min\{ p^{nr}, p^{(n-\rho)r+\delta_p}\}.
$$
In particular 
$K_{p^r}(\M;\a)=O_{\M}(1)$ if $\rho=n$.
\end{lemma}

This is established by 
Loxton \cite[Proposition 7]{loxton}, but is also a trivial consequence of earlier work of Smith \cite{smith}, which provides a precise equality for 
$K_{p^r}(\M;\a)$. 
We present a proof of Lemma \ref{lem:smith}, for completeness, the upper bound 
$K_{p^r}(\M;\a)\leq p^{nr}$ being trivial.
Given $\M$ as in the statement of the lemma, it follows from the theory of the Smith 
normal form that there exist unimodular integer matrices $\A,\B$ such that 
$$
 \A\M\B=\diag(M_1,\ldots,M_n),
$$
with $M_1,\ldots,M_n \in \ZZ$ satisfying $M_i\mid M_{i+1}$, for $1\leq i< n$.
In particular, since $\M$ has rank $\rho$, it follows that $M_i=0$ for $i>\rho$. 
Hence 
\begin{align*}
K_{p^r}(\M;\a)
&=
\#\{\x\bmod{p^r}: M_ix_i\equiv (\A\mathbf{a})_i \bmod{p^r}, ~(1\leq i \leq \rho)\}\\
&\leq 
 p^{(n-\rho)r+v_p(M_1)+\cdots +v_p(M_\rho)}.
\end{align*}
This completes the proof of Lemma \ref{lem:smith}, since  $\delta_p=
v_p(M_1)+\cdots +v_p(M_\rho)$.

We end this section by drawing a conclusion 
about the special case that $\M$ is non-singular, with $\rho=n$. 
Suppose that there exists a vector $\x$ counted by 
$K_{p^r}(\M;\0)$, but satisfying $p\nmid \x$.
Then it follows from our passage to the Smith normal form that in fact 
$r\leq v_p(\det \M)$.

\subsection{Geometry of $V$} 
\label{geometry}

In this section we consider the geometry of 
the varieties
$V\subset \PP^{n-1} $ defined by the common zero locus of 
two quadratic forms
$Q_1,Q_2\in \ZZ[x_1,\dots,x_n]$, specifically 
in the case that $V$ is non-singular.
Suppose that $Q_i$ has underlying symmetric matrix $\M_i$, 
with  $\M_{2}$  non-singular.
Let $D=D(Q_1,Q_2)$ be the discriminant of the pair $\{Q_1,Q_2\}$, which is a non-zero integer by assumption.  According to Gelfand, Kapranov and 
Zelevinsky \cite[\S 13]{GKZ}, $D$ has total degree $(n+2)2^{n+1}$ in the 
coefficients of $Q_1,Q_2$ and is equal to the discriminant of the bihomogeneous polynomial 
$$
F(\b,\x)=b_1Q_1(\x)+b_2Q_2(\x).
$$
We  write
\begin{equation}
  \label{eq:Mc}
  \M(\b)=b_1\M_1+b_2\M_2,
\end{equation}
for the underlying symmetric matrix.
It follows from  \cite[Lemma 1.13]{CT} that 
\begin{equation}
  \label{eq:rank}
\rank \M(\b) \geq n-1
\end{equation}
for any $[\b]\in \PP^1$. 
Furthermore, 
Reid's thesis \cite{reid} shows  that the binary form 
$P(\b)=\det \M(\b)$  has non-zero discriminant.

An important r\^ole in our work will be played by the 
dual variety  
$V^*\subset  {\PP^{n-1}}^*\cong \PP^{n-1}$ 
of $V$.   Consider the incidence relation
$$
I=\{(x,H)\in V\times {\PP^{n-1}}^*: H \supseteq \mathbb{T}_x(V)\},
$$
where $\mathbb{T}_x(V)$ denotes the tangent hyperplane to $V$ at $x$.
The projection $\pi_1: I \rightarrow V$ makes $I$ into a bundle over $V$
whose fibres are  subspaces of dimension $n - \dim V - 2=1$.
In particular $I$ is an irreducible variety of dimension $n - 2$.  Since $V^*$ is defined to be the image of 
the projection 
$\pi_2:I \rightarrow {\PP^{n-1}}^*$, it therefore follows that 
the dual variety $V^*$ is irreducible. 
Furthermore, since $I$ has dimension $n-2$ one might expect that $V^*$ is a hypersurface 
in ${\PP^{n-1}}^*$. This fact, which is valid for any irreducible non-linear 
complete intersection, is established by Ein \cite[Proposition~3.1]{ein}.
Elimination theory shows that the defining homogeneous polynomial may be taken to have coefficients in $\ZZ$.  
Finally, by work of Aznar \cite[Theorem~3]{aznar}, the degree of 
$V^*$ is
$4(n-2)$. 
Hence $V^*$ is defined by an equation $G=0$, where 
$
G\in \ZZ[x_1,\dots,x_n]
$ 
is an absolutely irreducible form of degree $4(n-2)$.

Given a prime $p$, which is sufficiently large in terms of the coefficients of $V$, 
the reduction  of $V$ modulo $p$
will inherit many of the basic properties enjoyed by $V$ as a variety over $\QQ$. 
In particular it will
continue to be a 
non-singular complete intersection of codimension $2$,
satisfying the property that \eqref{eq:rank} holds for any 
$[\b]\in \PP^1$, where now $\M_i$ is 
taken to be the matrix obtained after reduction modulo $p$ of the
entries. Furthermore we may assume that $p \nmid 2\det \M_2$ and that
the discriminant of the polynomial $P(\b)$ does not vanish
modulo $p$.   We will henceforth set 
$$
 \Delta_V =O(1)
$$
to be the product of all primes for which any one of these properties fails at that prime.

\subsection{The function $\rho(d)$}

In this section we establish Hypothesis-$\rho$ when $V$ is non-singular, where 
$\rho(d)=S_{d,1}(\mathbf{0})$, in the notation of \eqref{eq:S'}.
Note that 
$\rho^*(d)\leq \rho(d)$, where
$$
\rho^*(d)=\#\{\x\bmod{d}:  (d,\x)=1, ~Q_1(\x)\equiv Q_2(\x)\equiv 0 \bmod{d}\}.
$$
We proceed to establish the following result.

\begin{lemma}
\label{rho(d)}
Hypothesis-$\rho$ holds if $V$ is non-singular. 
\end{lemma}

\begin{proof}
We adapt an argument of Hooley \cite[\S 10]{nonary} used to handle the analogous situation for cubic hypersurfaces.
By multiplicativity it suffices to examine the case $d=p^r$ for a prime $p$ and $r\in \NN$. Extracting common factors between $\x$ and $p^r$, we see that  
\begin{equation}\label{m:1}
\rho(p^r)= \sum_{0\leq k <\frac{r}{2}} p^{kn} \rho^*(p^{r-2k}) +  p^{ (r-\lceil \frac{r}{2}\rceil)n  }.
\end{equation}
Using additive characters to detect the congruences gives
\begin{align*}
\rho^*(p^s)
&=\frac{1}{p^{2s}} \sum_{\b \bmod{p^s}} 
~~
\sideset{}{^{*}}
\sum_{\substack{\mathbf x \bmod{p^s}}}
e_{p^s}\left(b_1Q_1(\x)+b_2Q_2(\x)\right),
\end{align*}
where we recall that the notation $\sum^*$ means only 
$\x$ for which $p\nmid \x$ are of interest.
Extracting common factors between $p^s$ and $\b$ yields
\begin{align*}
\rho^*(p^s)
&=\frac{1}{p^{2s}}
\sum_{0\leq i<s} p^{in} S(s-i)  +p^{(n-2)s}\left(1-\frac{1}{p^n}\right),
\end{align*}
with 
$$
S(k)= \sideset{}{^{*}} \sum_{\b \bmod{p^k}} 
~~
\sideset{}{^{*}}
\sum_{\substack{\mathbf x \bmod{p^k}}}
e_{p^k}\left(F(\b,\x)\right),
$$
with $F(\b,\x)=b_1Q_1(\x)+b_2Q_2(\x)$.
We claim that 
$S(k)=O(1)$,
for any $k\in \NN$.   Once achieved, this implies  
that $\rho^*(p^s)=O(p^{(n-2)s})$. Inserting this into \eqref{m:1} gives
$\rho(p^r)=O(p^{(n-2)r})$,
which suffices for the lemma.

To analyse $S(k)$ we introduce a dummy sum over $a\in (\ZZ/p^k\ZZ)^*$ and replace $\b$ by $a\b$ to get
\begin{align*}
\phi(p^k)S(k)&= 
\sideset{}{^{*}} \sum_{a \bmod{p^k}} 
~~
\sideset{}{^{*}} \sum_{\b \bmod{p^k}} 
~~
\sideset{}{^{*}}
\sum_{\substack{\mathbf x \bmod{p^k}}}
e_{p^k}\left(aF(\b,\x)\right).
\end{align*}
Evaluating the resulting Ramanujan sum yields
\begin{equation}\label{m:3}
S(k)=\left(1-\frac{1}{p}\right)^{-1} \left\{ 
N(p^k)-p^{n+1}N(p^{k-1})
\right\},
\end{equation}
where
$N(p^k)$ is the number of $(\b,\x)\bmod{p^k}$, with $p\nmid \b$ and $p\nmid \x$, for which $p^k\mid F(\b,\x)$.  
We are therefore led to compare $N(p^k)$ with $N(p^{k-1})$, using an approach based on Hensel's lemma.

Let $\nabla F(\b,\x)=(Q_1(\x),Q_2(\x),b_1\nabla_\x Q_1(\x)+b_2\nabla_\x Q_2(\x)  )$, where 
$\nabla_\x$ means that the partial derivatives are taken with respect to the $\x$ variables. 
Using our alternative definition of the discriminant $D$ as the discriminant of $F$, we may view $D$ as the resultant of the $n+2$ quadratic forms appearing in 
$\nabla F(\b,\x)$. Writing $\y=(\b,\x)$, 
elimination theory therefore produces $n+2$ identities of the form
$$
Dy_i^N = \sum_{1\leq j\leq n+2} G_{ij}(\y) \frac{\partial F}{\partial y_i}, \quad (1\leq i\leq n+2),
$$
where $G_{ij}$ are polynomials with coefficients in $\ZZ$.
In particular, if $(\b,\x)\in\ZZ^{n+2}$ satisfies $p^m  \mid \nabla F(\b,\x)$, but   
$p\nmid \b$ and $p\nmid \x$, it follows that  
$m\leq v_p(D)$. Let us put $\delta=v_p(D)$.  

If $k\leq 2\delta+1$ then it trivially follows from \eqref{m:3} that $S(k)=O(1)$. 
If $k\geq 2\delta +2$, which we assume for the remainder of the argument, we will show   that $S(k)=0$.
Our work so far has shown that
$$
N(p^k)=\sum_{0\leq m\leq \delta} \#C_m(p^k),
$$
where $C_m(p^k)$ denotes the set 
of $\y=(\b,\x)\bmod{p^k}$, with $p\nmid \b$ and $p\nmid \x$, for which $p^k\mid F(\y)$ and $p^m  \| \nabla F(\y)$.
Given any $\y\in C_m(p^k)$ it is easy to see that 
\begin{align*}
F(\y+p^{k-m}\y')
&\equiv F(\y)+p^{k-m} 
\y'.
\nabla F(\y)
   \bmod{p^k}\\
&\equiv 0 \bmod{p^k},
\end{align*}
for any $\y'\in\ZZ^{n+2}$, with 
\begin{align*}
\nabla F(\y+p^{k-m}\y') -\nabla F(\y)
&\equiv 0 
   \bmod{p^{k-m}}\\
&\equiv 0 \bmod{p^{m+1}},
\end{align*}
Thus $C_m(p^k)$ consists of cosets modulo $p^{k-m}$. Moreover, 
$\y+p^{k-m}\y'\in C_m(p^{k+1})$ if and only if 
$$
p^{-k}F(\y)+p^{-m}\y'.
\nabla F(\y)\equiv 0
   \bmod{p},
$$
for which there are precisely $p^{n+1}$ incongruent solutions modulo $p$.
Hence 
$\#C_m(p^{k+1})=p^{n+1} \#C_m(p^k)$, which therefore shows that $S(k)=0$ in \eqref{m:3}.
This completes the proof of the lemma.
\end{proof}

\subsection{Treatment of bad $d$}

Returning briefly to $S(B)$ in \eqref{eq:main-sum}, we will need a separate argument to deal with the contribution from $\x$ for which $Q_2(\x)=0$ and $Q_1(\x)$ is divisible by 
large  values of  $d$ which share a common prime factor with $\Delta_V$.

To begin with we call upon joint work of the first author with Heath-Brown and Salberger \cite{bhbs}, which is concerned with uniform upper bounds for counting functions of the shape 
$$
M(f;B)=\#\{\t\in\ZZ^\nu:  |\t|\leq B, ~ f(\t)=0\},
$$
for polynomials $f\in\ZZ[t_1,\ldots,t_\nu]$ of degree $\delta\geq 2$.
Although the paper focuses on the situation for $\delta\geq 3$, the methods developed also permit a useful estimate in the case $\delta=2$. Suppose that $\nu=3$
and that the quadratic homogeneous part $f_0$ of $f$ is absolutely irreducible.  
Using 
 \cite[Lemmas 6 and 7]{bhbs} we can find a linear form $L\in \ZZ[t_1,t_2,t_3]$ of height $O(1)$ such that the intersection of the projective plane curves $f_0=0$ and $L=0$ consists of two distinct points. After eliminating one of the variables,  we are then free to apply 
 \cite[Lemma 13]{bhbs} to all the affine curves defined by $f=0$ and $L=c$, for each integer $c\ll B$. This gives the upper bound
$M(f;B)\ll B^{1+\ve}$ when $\nu=3$. 
According to \cite[Lemma 8]{bhbs}, we have therefore established the following result, which may be of independent interest.

\begin{lemma}
\label{m:5}
Let $\ve>0$, let $\nu\geq 3$ and let 
$f\in\ZZ[t_1,\ldots,t_\nu]$ be a quadratic polynomial with  absolutely irreducible quadratic homogeneous part.
Then we have 
$$M(f;B)\ll B^{\nu-2+\ve}.
$$
The implied constant in this estimate depends at most on $\nu$ and the choice of $\ve$.
\end{lemma}

We shall also require some facts about lattices and their successive
minima, as established by  Davenport 
\cite[Lemma 5]{Dav}.
Suppose that $\Lambda\subset \ZZ^n$ is a lattice of rank $r$ and determinant $\det(\Lambda)$.
Then there exists a minimal basis 
 $\ma{m}_1,\ldots,\ma{m}_r$ of $\Lambda$ such
that $|\ma{m}_i|$ is equal to the $i$th successive minimum $s_i$, 
for $1 \leq i \leq r$,  with the property that
whenever one writes $\y\in\Lambda$ as 
$$
\y=\sum_{i=1}^{r}\lambda_{i}\ma{m}_i,
$$ 
then $\lambda_{i}\ll  s_i^{-1}|\y|$, 
for $1 \leq i \leq r.$ Furthermore, 
$$
\prod_{i=1}^r s_i \ll \det \Lambda \le \prod_{i=1}^r s_i,
$$
and $1\leq s_1\leq \cdots \leq s_n$.

We now come to the key technical estimate in this section. 
Given any $d\in \NN$ and $B\geq 1$, we will need an auxiliary upper bound for the quantity
\begin{equation}\label{eq:def-Ne}
N_d(B)=\#\{ \x\in\ZZ^n:  |\x|\leq B, ~d\mid Q_1(\x), ~ Q_2(\x)=0 \}.
\end{equation}
Simple heuristics suggest that $N_d(B)$ should have order $d^{-1}B^{n-2}$. For our purposes we require an upper bound in which any power of $d$ is saved.

\begin{lemma}
\label{lem:m:4}
Let $\ve>0, d\in \NN$ and $n\geq 5$.
Assume 
$B\geq d$ and 
Hypothesis-$\rho$.
Then we have
$$
N_d(B)\ll \frac{B^{n-2+\ve}}{d^{\frac{1}{n}}} +dB^{n-3+\ve}.
$$ 
\end{lemma}

Note that this estimate is valid for any quadratic forms $Q_1,Q_2$ for which $Q_2$ is non-singular and the expected bound for $\rho(d)$ holds.  For our purposes the desired bound follows from Lemma \ref{rho(d)} when $V$ is non-singular.

\begin{proof}[Proof of Lemma \ref{lem:m:4}]
On extracting common factors between $\x$ and $d$ in $N_d(B)$, one quickly verifies that it suffices to prove the upper bound in the lemma for the quantity $N_d^*(B)$, in which the additional constraint $(d,\x)=1$ is added.  Breaking into residue classed modulo $d$, we see that
\begin{equation}\label{m:6}
N_d^*(B)=
\sideset{}{^{*}} \sum_{\substack{\bxi\bmod{d}\\  Q_1(\bxi)\equiv 0 \bmod{d} \\
 Q_2(\bxi)\equiv 0 \bmod{d}
}}
\#\{
 \x\in\ZZ^n:  |\x|\leq B, ~\x\equiv \bxi \bmod{d}, ~ Q_2(\x)=0
\}.
\end{equation}
Let us denote the set whose cardinality appears in the inner sum by $S_d(B;\bxi)$.
If $S_d(B;\bxi)=\emptyset$ then there is nothing to prove. Alternatively, suppose we are given $\x_0\in S_d(B;\bxi)$. Then any other vector in the set must be congruent to $\x_0$ modulo $d$.

Making the change of variables 
$\x=\x_0+d\y$ in $S_d(B;\bxi)$, we note that $|\y|<Y$, with $Y=2d^{-1}X$.  Furthermore, Taylor's formula yields
\begin{equation}\label{m:4}
\y.\nabla Q_2(\x_0) +dQ_2(\y)=0,
\end{equation}
since $Q_2(\x_0+d\y)=0$ and $Q_2(\x_0)=0$.
This equation implies that the $\y$ under consideration are forced to satisfy 
the congruence 
$\y.\nabla Q_2(\bxi)\equiv 0\bmod{d}$, since 
$\x_0\equiv \bxi \bmod{d}$.
Let us write $\a=\nabla Q_2(\bxi)$. 
Then it follows that 
$$
\#S_d(B;\bxi)\leq 1+ \#\{
 \y\in\Lambda_\a :  |\y|<Y , ~\mbox{\eqref{m:4} holds}\},
$$
where $\Lambda_\a=\{\y\in \ZZ^n:  \a.\y\equiv 0 \bmod{d}\}$.
This set defines an integer lattice of full rank and determinant 
$$
\det \Lambda_\a = \frac{d}{(d,\a)}.
$$
The conditions of summation in \eqref{m:6} demand that  $(d,\bxi)=1$.
It therefore follows from the remark at the end of \S \ref{s:congruences} that 
$p^j\ll 1$, whenever $j\in \NN$ and $p$ is a prime for which $p^j\mid (d,\nabla Q_2(\bxi))$.  
Thus $(d,\a)\ll 1$ and it follows that  
$\det \Lambda_\a \gg d.
$   

Let $\M$ denote the non-singular matrix formed from taking a minimal basis  $\m_1,\ldots,\m_n$ for $\Lambda_\a$. 
Making the change of variables $\y=\M\bla$, and recalling the properties of the minimal basis recorded above,  we see that 
$$
\#S_d(B;\bxi)\leq 1+ \#\{
 \bla\in\ZZ^n :  \mbox{$\lambda_i\ll s_i^{-1}Y$ for $1\leq i\leq n$} ,~ q(\bla)=0\},
$$
where $s_1,\ldots,s_n$ are the successive minima of $\Lambda_\a$ and 
$q(\bla)$ is obtained from 
\eqref{m:4} via substitution. 
In particular, it is clear that the quadratic homogeneous part $q_0$ of $q$ has underlying matrix $\M^T \M_2 \M$, which is non-singular.
We are therefore left with the task of counting integer solutions to a quadratic equation, which are constrained to lie in a lop-sided region. Furthermore, since we require complete uniformity in $d$, we want an upper bound in which the implied constant does not depend on the coefficients of $q$.

It being difficult to handle a genuinely lopsided region, we will simply fix the smallest variable and then allow the remaining vectors $\bla'=(\lambda_1,\ldots,\lambda_{n-1})$ to run over the full hypercube with side lengths $O(Y)$.  In this way we find that
$$
\#S_d(B;\bxi)\leq 1+ 
\sum_{
\substack{
t\ll s_n^{-1}Y
 }
}
\#\{
 \bla'\in\ZZ^{n-1} :  |\bla' |\ll Y, ~ q(\bla',t)=0\}.
$$
Viewed as a polynomial in $ \bla'$, the
quadratic homogeneous part of $q(\bla',t)$ is equal to $q_0(\bla',0)$.
This must have rank at least $n-2\geq 3$, since $q_0$ is non-singular and its rank 
cannot decrease by more than $2$ on any hyperplane.
In particular, $q_0(\bla',0)$ is absolutely irreducible.
We apply Lemma \ref{m:5} with $\nu=n-1$ and $f=q(\bla',t)$ to get
$$
\#S_d(B;\bxi)\ll Y^{n-3+\ve}\left(1+ \frac{Y}{s_n}\right).
$$
Now it follows from the general properties of the successive minima recorded above that $s_n\geq
(\det \Lambda_\a)^{\frac{1}{n}}\gg 
d^{\frac{1}{n}}$.  Recalling that $Y=2d^{-1}B$ and inserting this into \eqref{m:6}, we conclude that 
$$
N_d^*(B)\ll \rho(d)\left(\frac{B}{d}\right)^{n-3+\ve}
\left(1+ \frac{B}{d^{1+\frac{1}{n}}}\right).
$$
The conclusion of the lemma  therefore follows from Hypothesis-$\rho$.
\end{proof}

\section{Preliminary transformation of $S(B)$} 
\label{prelim}

In this section we initiate our analysis of $S(B)$ in \eqref{eq:main-sum}.
For any odd integer $M$ it is clear that $r(M)=0$ unless $M\equiv 1 \bmod{4}$. Hence our sum can be written
$$
S(B)=
\sum_{\substack{\mathbf x \in \mathbb Z^n\\ Q_1(\x)\equiv 1 \bmod{4} \\
Q_2(\x)=0}}r(Q_1(\mathbf
x)) W\left(\frac{\mathbf x}{B}\right).
$$
We proceed to open up the $r$-function in the summand.
Let $\{V_T(t)\}_{T}$ be a collection of smooth functions, with $V_T$
supported in the dyadic block $[T,2T]$, such that $\sum_TV_T(t)=1$
for $t\in[1,CB^2]$. The constant $C$ will be large enough depending on
$Q_1$ and $W$, so that $|Q_1(\mathbf x)|\leq C$ whenever $\mathbf
x\in \supp(W)$. We will neither specify the function $V_T$ nor
the indexing set for $T$.  However we will simply note that $T$  can be restricted
to lie in the interval  $[\frac{1}{2},2CB^2]$, and that there are $O(\log B)$
many functions in the collection. Moreover we will stipulate that  
$$
t^jV^{(j)}_T(t)\ll_j 1,
$$ 
for each integer $j\geq 0$.
For a positive integer $M\leq CB^2$ we may write
\begin{align*}
r(M)=4\sum_T\sum_{d\mid M}\chi(d)V_T(d).
\end{align*}
It follows that
$$
S(B)=4\sum_T\sum_{d}\chi(d)V_T(d)
\sum_{\substack{\mathbf x \in \mathbb Z^n\\
Q_1(\x)\equiv 1\bmod{4}\\ 
Q_1(\x)\equiv 0
    \bmod{d}\\Q_2(\x)=0}} 
W\left(\frac{\mathbf x}{B}\right)=4\sum_T S_T(B),
$$  
say.
Let $\a\in \left(\mathbb Z/4\mathbb Z\right)^n$ be such that
$Q_1(\a)\equiv 1 \bmod{4}$, and let $S_{T,\a}(B)$ be the part of
$S_T(B)$ which comes from $\x\equiv \a \bmod{4}$.

In the analysis of  $S_{T,\a}(B)$ we want to arrange things so that only values of $d$ satisfying 
$d\ll B$ occur. 
When $T\leq B$ this is guaranteed by the presence of the factor $V_T(d)$.
When $T>B$ we 
can use Dirichlet's hyperbola trick, since 
$\chi(Q_1(\x))=\chi(Q_1(\a))=1$, to get 
\begin{align*}
S_{T,\a}(B)=\sum_{d}\chi(d)
\sum_{\substack{\x \equiv \a \bmod{4}\\ Q_1(\x)\equiv 0 \bmod{d}\\Q_2(\x)=0}}
W\left(\frac{\mathbf x}{B}\right)V_T\left(\frac{Q_1(\x)}{d}\right).
\end{align*}
In this case too we therefore have $d\ll B$.
For notational simplicity we write
\begin{align}
\label{eq:W_d}
W_d\left(\mathbf y\right)=
\begin{cases} 
W\left(\mathbf y\right)V_T(d), &\mbox{if $T\leq B$,}\\
W\left(\mathbf y\right)V_T\left(\frac{B^{2}Q_1(\y)}{d}\right),
&\mbox{otherwise}.
\end{cases}
\end{align}
Here  
$W: \RR^{n}\rightarrow \RR_{\geq 0}$ is an infinitely
differentiable bounded function of compact support such that 
$Q_1(\x)\gg 1$ and $\nabla Q_1(\x)\gg 1$,
for some absolute implied constant, for every $\x \in
\supp(W)$.

As already indicated, the exponential sums 
\eqref{eq:S'} will be prominent in our work. 
We will face significant technical issues in dealing with large values of $d$ in $S_{T,\a}(B)$ 
which share prime factors with the constant 
 $\Delta_V$ that was introduced  at the close of \S \ref{geometry}.  
The following expression for $S(B)$ is now available.

\begin{lemma}
\label{S(B)}
Let  $\Xi$ be a parameter satisfying $1\leq \Xi\leq B$.
Then we have
$$
S(B)=
4\sum_T
\sum_{\substack{
\a\in (\mathbb Z/4\mathbb Z)^n\\
Q_1(\a)\equiv 1 \bmod{4}}}
\left(S_{T,\a}^\flat(B)+
S_{T,\a}^\sharp(B)\right),
$$
with 
\begin{align*}
S_{T,\a}^\flat(B)
&=\sum_{\substack{d=1\\ (d,\Delta_V^\infty)>\Xi}}^\infty
\chi(d)
\sum_{\substack{\x \equiv \a \bmod{4}\\ Q_1(\x)\equiv 0 \bmod{d}\\Q_2(\x)=0}}
W_d\left(\frac{\mathbf x}{B}\right),\\
S_{T,\a}^\sharp(B)
&=
\sum_{\substack{d=1\\ (d,\Delta_V^\infty)\leq \Xi}}^\infty
\chi(d)
\sum_{\substack{\x \equiv \a \bmod{4}\\ Q_1(\x)\equiv 0 \bmod{d}\\Q_2(\x)=0}}
W_d\left(\frac{\mathbf x}{B}\right).
\end{align*}
\end{lemma}

We will provide an upper bound for 
$S_{T,\a}^\flat(B)$ and an asymptotic formula for 
$S_{T,\a}^\sharp(B)$, always assuming that $\Xi$ satisfies $1\leq \Xi\leq B$.
The following result deals with the first task.

\begin{lemma}\label{lem:flat}
Let $\ve>0$ and assume 
Hypothesis-$\rho$.
Then we have 
$$S_{T,\a}^\flat(B)\ll \Xi^{-\frac{1}{n}}B^{n-2+\ve}+
\Xi B^{n-3+\ve}.
$$
\end{lemma}

\begin{proof}
Write
$e=(d,\Delta_V^\infty)$.  Then 
\begin{align*}
|S_{T,\a}^\flat(B)|
&\leq 
\sum_{\substack{e\mid \Delta_V^\infty\\ 
e>\Xi}}
\sum_{\substack{d=1}}^\infty
\sum_{\substack{\x \equiv \a \bmod{4}\\ Q_1(\x)\equiv 0 \bmod{de}\\Q_2(\x)=0}}
W_{de}\left(\frac{\mathbf x}{B}\right).
\end{align*}
By the properties of \eqref{eq:W_d}, only $d,e$ satisfying $de\ll B$ feature here.
Inverting the sums over $d$ and $\x$, we obtain 
\begin{align*}
S_{T,\a}^\flat(B)
&\ll
\sum_{\substack{e\mid \Delta_V^\infty\\ 
\Xi<e\ll B}}
\sum_{\substack{
|\x|\ll B\\
\x \equiv \a \bmod{4}\\ Q_1(\x)\equiv 0 \bmod{e}\\Q_2(\x)=0}}
\tau\left(\frac{Q_1(\x)}{e}\right),
\end{align*}
where $\tau$ is the divisor function.
Note that $Q_1(\x)\neq 0$, since $Q_1(\x)\equiv Q_1(\a)\equiv 1 \bmod{4}$, 
so that the inner summand is $O(B^\ve)$ by the trivial estimate for $\tau$. 
Hence we have 
\begin{equation}\label{eq:train}
S_{T,\a}^\flat(B)
\ll B^\ve
\sum_{\substack{e\mid \Delta_V^\infty\\ 
\Xi<e\leq cB}}
N_e(cB),
\end{equation}
for an absolute constant $c>0$, in the notation of \eqref{eq:def-Ne}.

We will make crucial use of the monotonicity property
$
N_e(cB)\leq N_d(cB)
$
for $d\mid e$.
Suppose that we have a factorisation
$\Delta_V=\prod_{i=1}^{t} p_i$. For $\mathbf{n}\in \ZZ_{\geq 0}^t$, let
$\mathbf{p}^{\mathbf{n}}=\prod_{i=1}^t p_i^{n_i}$. 
Consider a collection of integers $\mathcal B=\{\mathbf{p}^{\mathbf{n}}:
\mathbf{n}\in \ZZ_{\geq 0}^t\}$
and set 
$
\mathcal B(A_1,A_2)=\mathcal B\cap(A_1,A_2].
$
It follows from
\eqref{eq:scat} that  $\mathcal{B}$ contains $O(B^\ve)$ elements of order $B$.
In this new notation the sum in \eqref{eq:train} is over 
$e\in \mathcal B(\Xi,cB)$.  
We claim that 
\begin{align*}
S_{T,\a}^\flat(B)
&\ll B^\ve
\sum_{e\in \mathcal B(
\Xi,\Delta_V \Xi)} N_e(cB).
\end{align*}
Once achieved, the statement of the lemma will then follow from 
Lemma \ref{lem:m:4}.

By the monotonicity property, 
in order to establish the claim
it will suffice to show that every $e \in \mathcal B(\Xi, cB)$ has a divisor $e'\mid e$, with $e' \in
 \mathcal B(\Xi, \Delta_V \Xi)$. To see this we 
suppose that $e=\mathbf{p}^{\mathbf{n}}$ and consider the 
decreasing sequence of divisors of $e$. This sequence ends at $1$, and the ratio between any two consecutive members is bounded by $\Delta_V$ . Thus one of the divisors must lie in the range $(\Xi, \Delta_V \Xi]$, as required.  
This completes the proof of the lemma.
\end{proof}

Turning to $S_{T,\a}^\sharp(B)$, we 
now need a means of detecting the equation $Q_2(\x)=0$.
For any integer $M$
let
$$
\delta(M)=\begin{cases}
1,& \mbox{if $M=0$,}\\ 
0,&\mbox{otherwise}.
\end{cases}
$$
Our primary tool in this endeavour will be a version of the circle method developed by
Heath-Brown \cite{H}, based on work of Duke, Friedlander and Iwaniec \cite{DFI}.  
The starting point for this  is the following 
smooth approximation of $\delta$.

\begin{lemma}
\label{dfi}
For any $Q>1$ there is a positive constant $c_Q$, and a smooth
function $h(x,y)$ defined on $(0,\infty)\times\mathbb R$, 
such that
\begin{align*}
\delta(M)=\frac{c_Q}{Q^2}\sum_{q=1}^{\infty}\;\sideset{}{^{*}}\sum_{a
  \bmod{q}}e_q(aM)h\left(\frac{q}{Q},\frac{M} 
{Q^2}\right).
\end{align*}
The constant $c_Q$ satisfies $c_Q=1+O_N(Q^{-N})$ for any $N>0$. Moreover $h(x,y)\ll x^{-1}$ for all $y$, and 
$h(x,y)$ is non-zero only for $x\leq\max\{1,2|y|\}$. 
\end{lemma} 

In practice, to detect the equation $M=0$ for a sequence of integers
in the range $|M|<N/2$, it is 
logical to choose $Q=N^{\frac{1}{2}}$. 
We will use the above lemma to detect the equality $Q_2(\x)=0$ in
$S_{T,\a}^\sharp(B)$. Since we already have the modulus $d$ 
in the sum over $\x$ it is reasonable to use this modulus to reduce the
size of the parameter $Q$. Thus we replace 
the equality $Q_2(\x)=0$ by the congruence $Q_2(\x)\equiv 0 \bmod{d}$
and the equality $Q_2(\x)/d=0$. Then  
we have
\begin{align*}
S_{T,\a}^\sharp(B)&=
\sum_{\substack{d=1\\ (d,\Delta_V^\infty)\leq \Xi}}^\infty
\chi(d)
\sum_{\substack{\x \equiv \a \bmod{4}\\ Q_1(\x)\equiv 0\bmod{d}\\Q_2(\x)\equiv 0\bmod{d}}}
\delta\left(\frac{Q_2(\x)}{d}\right)W_d\left(\frac{\x}{B}\right)\\
&=
\sum_{\substack{d=1\\ (d,\Delta_V^\infty)\leq \Xi}}^\infty
\frac{\chi(d)c_Q}{Q^2}
\sum_{q=1}^{\infty}\;\sideset{}{^{*}}\sum_{a\bmod{q}}
\sum_{\substack{\x \equiv \a \bmod{4}\\ Q_1(\x)\equiv 0\bmod{d}\\Q_2(\x)\equiv 0\bmod{d}}}
\hspace{-0.4cm}
e_q\left(\frac{aQ_2(\x)}{d}\right)h\left(\frac{q}{Q},\frac{Q_2(\x)}{dQ^2}\right)W_d\left(\frac{\x}{B}\right).
\end{align*}
We shall  make the choice
$$
Q= 
\frac{B}{\sqrt{d}}.
$$
Since $d\ll B$, it follows that  $Q\gg \sqrt{B}$.

With our choice of $Q$ made we remark that 
the size of the full modulus $qd$ 
is typically of order $B^{\frac{3}{2}}$. Since this is much smaller than the
square of the 
length of each $x_i$ summation, it will be be profitable to use the
Poisson summation formula on the sum over $\x$.

\begin{lemma}
\label{psum}
For any $N>0$ we have
$$
S_{T,\a}^\sharp(B)=\left(1
+O_N(B^{-N})\right)
\frac{B^{n-2}}{4^n}\sum_{\m\in \mathbb
  Z^n}
  \sum_{\substack{d=1\\ (d,\Delta_V^\infty)\leq \Xi}}^\infty
\hspace{-0.2cm}
\frac{\chi(d)}{d^{n-1}}
  \sum_{q=1}^\infty\frac{1}{q^n} 
T_{d,q}(\m)I_{d,q}(\m),
$$
where 
$$
T_{d,q}(\m)=\sideset{}{^{*}}\sum_{a\bmod{q}}
\sum_{\substack{\mathbf k \bmod{4dq}\\
\k \equiv \a \bmod{4}\\ Q_1(\k)\equiv 0\bmod{d}\\Q_2(\k)\equiv 0\bmod{d}}}
e\left(\frac{4aQ_2(\k)+\m.\k}{4dq}\right)
$$
and 
$$
I_{d,q}(\m)=\int_{\mathbb R^n}
h\left(\frac{q}{Q},\frac{B^2Q_2(\y)}{dQ^2}\right)W_d(\y)
e_{4dq}(-B\m.\y)\d \mathbf y.
$$
\end{lemma}

\begin{proof}
Splitting the sum over $\mathbf x$ into residue classes 
modulo $4dq$, we get that the inner sum over $\x$ in our expression for $S_{T,\ma{a}}^\sharp(B)$ is given by
\begin{align*}
\sum_{\substack{\mathbf k \bmod{4dq}\\
\k \equiv \a \bmod{4}\\ Q_1(\k)\equiv 0\bmod{d}\\Q_2(\k)\equiv 0\bmod{d}}}
e\left(\frac{aQ_2(\k)}{qd}\right)
\sum_{\x\in \mathbb Z^n}f(\x),
\end{align*}
where 
$$
f(\x)=h\left(\frac{q}{Q},\frac{Q_2(\k+ 4dq\x)}{dQ^2}\right)W_d\left(\frac{\k+4dq\x}{B}\right).
$$ 
The  Poisson summation
formula yields
\begin{align*}
\sum_{\x\in \mathbb Z^n}f(\x)=\sum_{\m\in \mathbb Z^n}\hat f(\m),
\end{align*}
where
\begin{align*}
\hat f(\m)&=\int_{\mathbb R^n}f(\mathbf y)e(-\m.\mathbf y)\d\mathbf y\\
&=\left(\frac{B}{4dq}\right)^ne_{4dq}(\m.\k)\int_{\mathbb R^n}
h\left(\frac{q}{Q},\frac{B^2Q_2(\y)}{dQ^2}\right)W_d\left(\y\right)
e_{4dq}(-B\m.\y)\d\mathbf y.
\end{align*}
The lemma follows on 
rearranging 
and noting that 
$c_Q=
1
+O_N(B^{-N})$ and $Q^2=B^2/d$.
\end{proof} 

In this and the next few sections, we will analyse in detail the
exponential  
sum $T_{d,q}(\m)$ which appears in Lemma \ref{psum}. 
We start with a multiplicativity relation which reduces the problem to
analysing the sum for a prime power modulus. Observe that $d$
is necessarily odd, but $q$ can be of either parity.  
For any $d,q\in \NN$ we recall the definition \eqref{eq:S'} of 
$S_{d,q}(\m)$, and for any non-negative integer $\ell$ define 
\begin{align}
\label{eq:Sell}
S^{\pm}_{1,2^{\ell}}
(\m)=
\sideset{}{^{*}}\sum_{a\bmod{2^{\ell}}}
\sum_{\substack{\mathbf k \bmod{2^{2+\ell}}\\
\k\equiv \pm \ma{a}\bmod{4}}}
e_{2^{2+\ell}}\left(4aQ_2(\k)+\m.\k \right).
\end{align}
We  note that if $h\in \NN$ is coprime to $d$ and $q$ then 
$S_{d,q}(h\m)=S_{d,q}(\m)$. The  following result is now available.

\begin{lemma}\label{lem:mult1}
For $q=2^{\ell}q'$,  with $q'$ odd,  we have
\begin{align*}
T_{d,q}(\m)=S_{d,q'}(\m)S^{\chi(dq')}_{1,2^{\ell}}(\m).
\end{align*}
\end{lemma}

\begin{proof}
Set
\begin{align*} 
\k = \k' 2^{\ell+2}\overline{2^{\ell+2}}+ \k'' dq'\overline{dq'},
\quad
a= a'2^{\ell}\overline{2^{\ell}} + a''q'\overline{q'},
\end{align*}
where $\k'\bmod{dq'}$, $\k'' \bmod{2^{\ell+2}}$, $a'\bmod{q'}$, and $a''\bmod{2^{\ell}}$. 
The conditions on $\k$ then translate into 
$
\k''\equiv \a \bmod{4}$, $Q_1(\k')\equiv 0\bmod{d}$ and
$Q_2(\k')\equiv 0\bmod{d}.$
Furthermore, we have
$$
e\left(\frac{4aQ_2(\k)+\m.\k}{4dq}\right)=
e\left(\frac{(4a'Q_2(\k')+\m.\k')\overline{2^{\ell+2}}}{dq'}\right)
e\left(\frac{(4a''Q_2(\k'')+\m.\k'')\overline{dq'}}{2^{\ell+2}}\right).
$$
The sum over $a'$ and $\k'$ gives $S_{d,q'}(\m)$ 
after a change of variables. A similar change of variables in 
$a''$ and $k''$ gives $S^{\pm}_{1,2^{\ell}}(\m)$, where the sign is given by $\chi(dq')$. 
\end{proof}

In a similar spirit we can prove the following multiplicativity property for the sum \eqref{eq:S'}. 

\begin{lemma}\label{lem:mult2}
For $d=d_1d_2$ and $q=q_1q_2$, with $(d_1q_1,d_2q_2)=1$, we have
\begin{align*}
S_{d,q}(\m)=S_{d_1,q_1}(\m)S_{d_2,q_2}(\m).
\end{align*}
\end{lemma}

This result reduces the problem of estimating $S_{d,q}(\m)$ into three distinct cases. Accordingly, for $d,q\in \NN$ we define  the sums 
$$
\mathcal
Q_{q}(\m)=S_{1,q}(\m),\quad
\mathcal D_{d}(\m)=S_{d,1}(\m),\quad 
 \mathcal
M_{d,q}(\m)=S_{d,q}(\m),
$$ 
the latter sum only being of interest when $d$ and $q$ exceed $1$ and
are constructed from the same set of primes. 
The analysis of these sums will be the focus of \S \ref{sec:qsum}, \S \ref{sec:dsum1} and \S \ref{mcs}, respectively.
For the moment we content ourselves with recording the crude upper bound 
\begin{equation}
\label{eq:Sell-upper}
S^{\pm}_{1,2^{\ell}} (\m)\ll 2^{\ell(\frac{n}{2}+1)},
\end{equation}
for \eqref{eq:Sell}, whose truth will be established in the following section.

\medskip

We close this section by presenting some facts  concerning the exponential integral 
$I_{d,q}(\m)$
which appears in Lemma \ref{psum},  recalling the definition  \eqref{eq:W_d} of 
$W_d\left(\mathbf y\right)$.
The properties of $h$ recorded in Lemma \ref{dfi} ensure that $q\ll Q$
when $I_{d,q}(\m)$ is non-zero. Likewise the properties of $W_{d}$ imply that $d\ll B$ under the same hypothesis. 
The underlying weight function $W$ has bounded
derivatives  
$$
\frac{\partial^{i_1+\cdots+i_n}}{\partial y_1^{i_1}\cdots\partial
  y_n^{i_n}}W(\y)\ll_{i_1,\dots,i_n}1, 
$$
and the function $V_T$ satisfies $t^jV_T^{(j)}(t)\ll_j 1$. It therefore follows that
$$
\frac{\partial^{i_1+\cdots+i_n}}{\partial y_1^{i_1}\cdots\partial
  y_n^{i_n}}W_d(\y)\ll_{i_1,\dots,i_n}1,
$$
since $Q_{1}(\y)$ has order of magnitude $1$ for every $\y\in \supp(W)$.

In the notation of \cite[\S 7]{H} we have 
\begin{equation}\label{eq:I*}
I_{d,q}(\m)=I_r^{*}(\mathbf v)=\int_{\mathbb R^n}
h\left(r,G(\y)\right)\omega\left(\y\right)
e_{r}(-\mathbf v.\y)\d\y,
\end{equation}
where 
$$
r=\frac{q}{Q}, \quad \mathbf v= \frac{B\m}{4dQ}, \quad G(\y)=\frac{B^2Q_2(\y)}{dQ^2}=
Q_2(\y), \quad 
\omega (\y)= W_d(\y). 
$$
We have
$$
\frac{\partial^{i_1+\cdots+i_n}}{\partial y_1^{i_1}\cdots\partial y_n^{i_n}}G(\y)\ll_{i_1,\dots,i_n}1,
\quad
\frac{\partial^{i_1+\cdots+i_n}}{\partial y_1^{i_1}\cdots\partial
  y_n^{i_n}}\omega(\y)\ll_{i_1,\dots,i_n}1. 
$$
Using these bounds and integration by parts, as in \cite[\S 7]{H}, 
we obtain the following bound. 

\begin{lemma}
\label{ubI_q}
For $\m \neq \mathbf 0$ and any $N\geq 0$, we have
$$
I_{d,q}(\m) \ll_{N} \frac{Q}{q}\left(\frac{dQ}{B|\m|}\right)^N.
$$
\end{lemma}

As a consequence we get that $\m$ with $|\m|>dQB^{-1+\varepsilon}$ will make a negligible 
contribution in our analysis of $S_{T,\ma{a}}^\sharp(B)$.
For $\m$ with $0<|\m|\leq dQB^{-1+\varepsilon}$ we need a more refined bound.

\begin{lemma}
\label{ubI_q2}
For $0<|\m| \leq dQB^{-1+\varepsilon}=\sqrt{d}B^{\ve}$ and $q\ll Q=B/\sqrt{d}$, we have
\begin{align*}
\frac{\partial^{i+j}}{\partial d^i\partial q^j}I_{d,q}(\m) &\ll d^{-i}q^{-j}\left|\frac{B\m}{dq}\right|^{1-\frac{n}{2}} B^{\varepsilon},
\end{align*}
for any $i,j\in\{0,1\}$.
\end{lemma} 
\begin{proof}
When $i=0$ this result follows from a closer study of the behaviour of the function $h(x,y)$, 
and is due to
Heath-Brown \cite[\S\S 4--8]{H}.  
Let us suppose that $i=1$. After a change of variables we have
$$
I_{d,q}(\m)=d^n\int_{\mathbb R^n}
h\left(\frac{q\sqrt{d}}{B},d^2Q_2(\y)\right)W_d\left(d\y\right)
e_{4q}(-B\m.\y)\d\mathbf y.
$$
We proceed to take the derivative with respect to $d$. The  right hand side is seen to be 
\begin{align*}
\frac{n}{d}I_{d,q}(\m)+d^n\int_{\mathbb R^n}
g_d(\y)  e_{4q}(-B\m.\y)\d\mathbf y,
\end{align*}
where if $h^{(1)}(x,y)=\frac{\partial}{\partial x}h(x,y)$ and $h^{(2)}(x,y)=\frac{\partial}{\partial y}h(x,y)$, then
\begin{align*}
g_d(\y)=~&
\frac{q}{2B\sqrt{d}}h^{(1)}\left(\frac{q\sqrt{d}}{B},d^2Q_2(\y)\right)W_d\left(d\y\right)\\
&
+2dQ_2(\y)h^{(2)}\left(\frac{q\sqrt{d}}{B},d^2Q_2(\y)\right)W_d\left(d\y\right)
+h\left(\frac{q\sqrt{d}}{B},d^2Q_2(\y)\right)\frac{\partial}{\partial d} W_d\left(d\y\right).
\end{align*}
Let $W^{(1)}(\y)=\y.\nabla W(\y)$. 
One finds that 
$$
\frac{\partial}{\partial d} W_d\left(d\y\right)=
 \frac{1}{d}W^{(1)}\left(d\mathbf y\right)V_T(d)+W\left(d\mathbf y\right)V_T'(d),
 $$
if $T\leq B$, and 
$$
\frac{\partial}{\partial d} W_d\left(d\y\right)
=\frac{1}{d}W^{(1)}\left(d\mathbf y\right)V_T\left(B^2dQ_1(\y)\right)+W\left(d\mathbf y\right)V_T'\left(B^2dQ_1(\y)\right)B^2Q_1(\y),
$$
otherwise.
Hence
\begin{align*}
\frac{\partial}{\partial d} W_d\left(d\y\right)=\frac{1}{d}W_{1,d}\left(d\y\right),
\end{align*}
where the new function $W_{1,d}$ has the same analytic behaviour as $W_d$. Another change of variables now yields
\begin{align*}
\frac{\partial}{\partial d}I_{d,q}(\m)
=~&\frac{n}{d}I_{d,q}(\m)+\frac{1}{2d}\int_{\mathbb R^n}\frac{q\sqrt{d}}{B}h^{(1)}\left(\frac{q\sqrt{d}}{B},Q_2(\y)\right)W_d\left(\y\right)e_{4dq}(-B\m.\y)\d\mathbf y\\
&+\frac{2}{d}\int_{\mathbb R^n}h^{(2)}\left(\frac{q\sqrt{d}}{B},Q_2(\y)\right)W_{2,d}\left(\y\right)e_{4dq}(-B\m.\y)\d\mathbf y\\
&+\frac{1}{d}\int_{\mathbb R^n}h\left(\frac{q\sqrt{d}}{B},Q_2(\y)\right)W_{1,d}\left(\y\right)e_{4dq}(-B\m.\y)\d\mathbf y,
\end{align*}
where $W_{2,d}(\y)=W_d(\y)Q_2(\y)$. The last three integrals can be compared with $I_{d,q}(\m)$, and the lemma now follows using the bounds in the statement of the lemma for $i=0$.
\end{proof}

\section{Analysis of $\cQ_q(\m)$}\label{sec:qsum}

The aim of this section is to collect together everything we need to
know about the sums
$$
\cQ_q(\m)=\sideset{}{^{*}}\sum_{\substack{a\bmod{q}}} \sum_{\k\bmod{q}} e_q(aQ_2(\k)+\m.\k),
$$
for given $\m \in \ZZ^n$.
This sum appears very naturally when the circle method is
employed to analyse quadratic forms. Let $\M$ be
the underlying symmetric $n\times n$ integer matrix for a quadratic form $Q$,
so that $Q(\k)=\k^T\M\k$. 
We begin with an easy upper bound for the inner sum in $\cQ_{q}(\m)$ when $q$ is a prime power.

\begin{lemma}
\label{lem:gauss-sum-bound}
For any quadratic form $Q(\x)=\x^T\M\x$, we have
$$
\left|\sum_{\mathbf k \bmod{p^r}}e_{p^r}\left(Q(\k)+\m.\k\right)\right|\leq p^{\frac{nr}{2}}
\sqrt{K_{p^r}(2\M;\0)},
$$
in the notation of \eqref{eq:2.1}.
\end{lemma}

\begin{proof}
Cauchy's inequality implies that the square of the left hand side is not greater than 
$$
\sum_{\x,\y\bmod{p^r}}
e_{p^r}\big((Q(\x)-Q(\y))+\m.(\x-\y)\big).
$$
Substituting $\x=\y+\z$ we see that the summand is equal to
$
e_{p^r}(\m.\z)e_{p^r}(Q(\z)+2\y^T \M\z).
$
The sum over $\y$ vanishes unless $p^{r}\mid 2\M\z$, in which case it
is given by $p^{nr}e_{p^r}(Q(\z))$. 
The result now follows by executing the sum over
$\z$ trivially.
\end{proof}

We apply Lemma \ref{lem:gauss-sum-bound} to estimate $\mathcal{Q}_{q}(\m)$.   Since $Q_2$ is non-singular 
it follows from Lemma~\ref{lem:smith} that there is an absolute constant $c\geq 1$ such that 
$
K_{p^r}(2\M_2;\ma{0})\leq c,
$
for any prime power $p^r$. Moreover one can take $c=1$ when $p\nmid 2\det \M_2$.
On summing trivially over $a$ one deduces that
$
|\mathcal Q_{p^r}(\m)|\leq  \sqrt{c} p^{\left(\frac{n}{2}+1\right)r}, 
$
for any prime power $p^r$.  Applying 
Lemma~\ref{lem:mult2} therefore yields 
\begin{equation}
\label{cor:Q_p^r}
\mathcal Q_{q}(\m)\ll q^{\frac{n}{2}+1}.
\end{equation}
Likewise  \eqref{eq:Sell-upper} is 
an easy consequence of Lemma \ref{lem:smith} and  Lemma \ref{lem:gauss-sum-bound} when $p=2$.

Using quadratic Gauss sums,  
it is possible to prove explicit formulae for $\cQ_{p^{r}}(\m)$
when the prime $p$ is large
enough. The oscillation in the sign of these sums will give
cancellation in the sum over $q$ in Lemma \ref{psum} which will be
crucial for handling  $n= 7$.  
Let $Q(\x)$ be a quadratic form with associated matrix $\M$. We
write $Q^*(\x)$ for the adjoint quadratic form with underlying
matrix $(\det \M)\M^{-1}$.    For any odd prime $p$ let
$$
\varepsilon(p)=
\begin{cases}
1, &\mbox{if $p\equiv 1 \bmod{4}$,}\\
i, &\mbox{if $p\equiv 3 \bmod{4}$,}
\end{cases} 
$$
and let $\chi_{p}(\cdot)$ denote the Legendre symbol
$(\frac{\cdot}{p})$. 
We may now record the following formula.

\begin{lemma}
\label{gs}
Let $p$ be a prime with $p\nmid 2\det \M$. Then we have
\begin{align*}
\sum_{\k \bmod{p^r}}e_{p^r}(Q(\k)+\m. \k)=
\begin{cases} 
p^{\frac{nr}{2}}e_{p^r}(-\overline{4\det \M}Q^*(\m)), &
\mbox{if $r$ is even},\\
p^{\frac{nr}{2}}\chi_p(\det
\M)\varepsilon(p)^ne_{p^r}(-\overline{4\det \M}Q^*(\m)),
&\mbox{if $r$
  is odd}. 
\end{cases}
\end{align*}
\end{lemma}

\begin{proof} 
Since $p$ is odd there exists a $n\times n$ matrix $\ma U$ with
integer entries and $p\nmid \det \ma U$ such that $\ma{U}^{T}\M\ma{U}$
is diagonal modulo $p^{r}$. Hence in proving the lemma we may restrict
ourselves to  diagonal forms  
$
Q(\x)=\al_1x_1^2+\cdots+\al_nx_n^2,
$
with $\M=\diag(\al_1,\dots,\al_n)$. In this case we have
$$
Q^*(\x)=\det \M\left(\frac{x_1^2}{\al_1}+\cdots+\frac{x_n^2}{\al_n}\right),
$$
where $\det \M=\al_1\cdots \al_n$. 

Let $S$ denote the sum appearing on the left hand side in the statement of the lemma.
Then
\begin{align*}
S=\prod_{i=1}^n\Bigl\{\sum_{k \bmod{p^r}}e_{p^r}(\al_i k^2+m_ik)\Bigr\}.
\end{align*}
Since $p\nmid 2\al_i$, we can complete the square. This yields
\begin{align*}
\sum_{k \bmod{p^r}}e_{p^r}(\al_i k^2+m_ik)=e_{p^r}(-\overline{4\al_i}m_i^2)
\sum_{k \bmod{p^r}}e_{p^r}(\al_i k^2).
\end{align*}
The last sum is the quadratic Gauss sum, which satisfies
\begin{align*}
\sum_{k (\text{mod}\;p^r)}e_{p^r}(\al_i k^2)=
\begin{cases}p^{\frac{r}{2}}, &\mbox{if $r$ is even,}\\
\chi_p(\al_i)\varepsilon(p)p^{\frac{r}{2}}, &\mbox{if $r$ is odd.}
\end{cases}
\end{align*}
The lemma follows on substituting this into the above expression for
$S$.\end{proof}

Lemma \ref{gs} directly yields an 
explicit evaluation of the  sum $\mathcal Q_{p^r}(\m)$ when
the prime $p$ is sufficiently large. To state the outcome of this let 
$$
c_{p^r}(a)=\sideset{}{^{*}}\sum_{x\bmod{p^r}}e_{p^r}\left(ax\right)
=\sum_{d\mid (p^r,a)} d\mu\left(\frac{p^r}{d}\right)
$$
be the Ramanujan sum and let
$$
g_{p^r}(a)=\sum_{x\bmod{p^r}}\chi_p(x)e_{p^r}\left(ax\right)
$$
be the Gauss sum. For the former we will make frequent use of the
fact that 
$c_{p^{r}}(ab)=c_{p^{r}}(a)$ for any $b$ coprime to $p$, and
$c_{p^{r}}(a_{1})=c_{p^{r}}(a_{2})$ whenever $a_{1}\equiv
a_{2}\bmod{p^{r}}$. Moreover, we have the obvious inequality 
$|c_{p^r}(a)|\leq (p^r,a)$.

It follows from Lemma \ref{gs} that
$$
\mathcal Q_{p^r}(\m)=p^{\frac{nr}{2}}\sideset{}{^{*}}\sum_{a\bmod{p^r}}\begin{cases} 
e_{p^r}(-\overline{4a\det \M_2}
Q^*_2(\m)),&\mbox{if $r$ is even,}\\
\chi_p(\det \M_2)\chi_p(a)^n\varepsilon(p)^ne_{p^r}(-\overline{4a\det
  \M_2}Q^*_2(\m)), &\mbox{if $r$ is odd},
\end{cases}
$$
if $p\nmid 2\det \M$.
The following lemma now follows from  executing the sum over $a$.

\begin{lemma}
\label{expQ}
Let $p$ be a prime with $p\nmid 2\det \M_{2}$. Then for even $n$ we have
$$
\mathcal Q_{p^r}(\m)=\varepsilon(p)^{nr}\chi_p(\det
\M_2)^rp^{\frac{nr}{2}}c_{p^r}\left(Q^*_2(\m)\right). 
$$
For odd $n$ we have
$$
\mathcal Q_{p^r}(\m)=\begin{cases}
p^{\frac{nr}{2}}c_{p^r}\left(Q^*_2(\m)\right), &
\mbox{if $r$ is even},\\
\varepsilon(p)^n \chi_p(-1)p^{\frac{nr}{2}}g_{p^r}(Q^*_2(\m)), &
\mbox{if $r$ is odd}.
\end{cases}
$$
\end{lemma}

Let 
\begin{equation}\label{eq:NM}
N=\begin{cases}
2\det \M_2 Q^*_2(\m), & \mbox{if $Q^*_2(\m)\neq 0$,}\\ 
2\det \M_2 , &\mbox{otherwise}.
\end{cases}
\end{equation}
We now turn to the average order of $\mathcal{Q}_q(\m)$, as one sums
over $q$ coprime to $M$ for some fixed $M\in \NN$ divisible by $N$. 
For this we will use Perron's formula unless $n$ is even and $Q_{2}^{*}(\m)\neq 0$, 
a case that can be handled trivially as follows.

\begin{lemma}\label{lem:q-triv}
Let $M \in \NN$ with $N\mid M$ and let $\ve>0$. Assume that $n$ is even and $Q_{2}^{*}(\m)\neq 0$. 
Then we have 
$$ 
\sum_{\substack{q\leq x\\ (q,M)=1}}|\cQ_{q}(\m)|
\ll
x^{\frac{n}{2}+1+\ve}M^{\ve}.
$$ 
\end{lemma}

\begin{proof}
Combining Lemma \ref{expQ}  with the multiplicativity relation Lemma \ref{lem:mult2} we obtain
$$
\sum_{\substack{q\leq x\\ (q,M)=1}}|\cQ_{q}(\m)|
\leq x^{\frac{n}{2}} 
\sum_{\substack{q\leq x\\ (q,M)=1}}|c_{q}(Q_{2}^{*}(\m))|.
$$
The lemma is therefore an  easy consequence of the 
inequality $|c_q(a)|\leq (q,a)$ satisfied by  the Ramanujan sum.
\end{proof}

Let $\chi$ be 
a non-principal Dirichlet character with conductor $c_\chi$. 
It will be convenient to recall some preliminary facts concerning the size of Dirichlet $L$-functions $L(s,\chi)$ in the critical strip.  
We begin by recalling the convexity bound 
\begin{equation}\label{eq:convex}
L(\sigma+it,\chi) \ll (c_\chi |t|)^{\frac{1-\sigma}{2}+\ve},
\end{equation}
for any $\sigma\in [0,1]$ and $|t|\geq 1$.
Next we claim that 
\begin{align}
\label{eq:l-series-bound}
\int_{\frac{1}{2}-iT}^{\frac{1}{2}+iT}|L(s,\chi)|^2\frac{\d s}{|s|}\ll
c_{\chi}^{\frac{7}{16}+\ve}T^{\varepsilon}.
\end{align}
In order to show this 
we break the  integral into dyadic blocks, deducing that it is
dominated by
$$
\sum_{\substack{
Y\;\text{dyadic}\\
\frac{1}{2}<Y\leq T
}}\frac{1}{1+Y}\int_{Y}^{2Y}\left|L\left(\frac{1}{2}+it,\chi\right)\right|^2\d
t. 
$$ 
For small values of $Y$ we use  
Heath-Brown's \cite{hb-hybrid}  
hybrid bound $L(\frac{1}{2}+it,\chi)\ll (c_\chi |t|)^{\frac{3}{16}+\ve}$, for $|t|\geq 1$, 
to get
$$
\frac{1}{1+Y}\int_{Y}^{2Y}\left|L\left(\frac{1}{2}+it,\chi\right)\right|^2
\d t\ll
c_{\chi}^{\frac{3}{8}+\ve}\sqrt{Y}.
$$ 
For  larger values of $Y$ we use the approximate functional equation to replace
the $L$-value by a series of length $\sqrt{c_{\chi}Y}$, and then use the
mean value theorem for Dirichlet polynomials (see Iwaniec and Kowalski \cite[Theorem 9.1]{HIEK}, for example). This gives
$$ 
\frac{1}{1+Y}\int_{Y}^{2Y}\left|\sum_{n\leq
\sqrt{c_{\chi}Y}T^{\ve}}\frac{\chi(n)}{\sqrt{n}}n^{-it}\right|^2\d t\ll
\left(1+\sqrt{\frac{c_{\chi}}{Y}}\right)T^{\ve}.
$$ Summing over all dyadic blocks, we easily arrive  at the claimed bound
\eqref{eq:l-series-bound}.

For $s\in \CC$ let 
$\sigma=\Re(s)$.
Returning now to the application of Perron's formula, we set
$$ 
\xi_M(s;\m)=\sum_{(q,M)=1}\frac{\cQ_{q}(\m)}{q^s}.
$$ 
By \eqref{cor:Q_p^r} this series is absolutely convergent for $\sigma>\frac{n}{2}+2$.
When $n$ is even and $Q_2^*(\m)\neq 0$ it is  
absolutely convergent for $\sigma>\frac{n}{2}+1$, by Lemma \ref{lem:q-triv}. 
For any  
$x-\frac{1}{2}\in \ZZ$ and $T>0$  we obtain
\begin{align}
\label{perron}
\sum_{\substack{q\leq x\\ 
(q,M)=1}}\cQ_{q}(\m)=\frac{1}{2\pi
i}\int_{c-iT}^{c+iT}\xi_M(s;\m) x^s\frac{\d s}{s}+O\left(
\frac{x^c}{T}\right),
\end{align}
where $c>\frac{n}{2}+2$.
We will take $T$ large enough in terms of $x$ and $|\m|$
so that the error term in the formula is
negligible. The analytic nature of the $L$-series can be revealed
using the explicit formulae that we enunciated in Lemma \ref{expQ} and depends on the parity of $n$.
For even $n$ we get
\begin{equation}
\label{eq:even-n}
\xi_M(s;\m)=\prod_{p\nmid M}\left\{\sum_{r=0}^{\infty}
\frac{\chi_p(\det \M_2)^r\varepsilon(p)^{nr}c_{p^r}\left(Q^*_2(\m)\right)}{p^{\left(s-\frac{n}{2}\right)r}}\right\}.
\end{equation}
For odd $n$ we get
\begin{equation}
\label{eq:odd-n}
\xi_M(s;\m)=\prod_{p\nmid M}\left\{\sum_{r\;\text{even}}
\frac{c_{p^r}\left(Q^*_2(\m)\right)}{p^{\left(s-\frac{n}{2}\right)r}}+
\chi_p(-1)\varepsilon(p)^n\sum_{r\;\text{odd}}
\frac{g_{p^r}\left(Q^*_2(\m)\right)}{p^{\left(s-\frac{n}{2}\right)r}}\right\}.
\end{equation}
The following result handles the case in which $Q_{2}^{*}(\m)=0$.

\begin{lemma}\label{lem:0}
Let $M \in \NN$ with $N\mid M$ and let $\ve>0$. Assume that
$Q^*_2(\m)=0$.   Then we have 
$$ 
\sum_{\substack{q\leq x\\ (q,M)=1}}\cQ_{q}(\m)\ll 
\begin{cases}
x^{\frac{n+3}{2}+\ve}M^{\ve},  &
\mbox{if $(-1)^{\frac{n}{2}}\det \M_2\neq \square$},\\
x^{\frac{n}{2}+2},  & 
\mbox{if $(-1)^{\frac{n}{2}}\det \M_2=\square$.}
\end{cases}
$$ 
\end{lemma}

Here, and after, for  any complex number $z$ we write $z=\square$ if and only if there 
exists an integer   
$j$ such that $z=j^{2}.$  Thus 
 the sum
in question is bounded by  
$O(x^{\frac{n+3}{2}+\ve}M^{\ve})$ when $n$ is odd since it is then
impossible for  
$(-1)^{\frac{n}{2}}\det \M_2$ to be the square of an integer.

\begin{proof}[Proof of Lemma \ref{lem:0}]
The second part of the lemma is a trivial consequence of \eqref{cor:Q_p^r} and the triangle inequality. Turning to the first part we begin by supposing that  $n$ is even and 
$(-1)^{\frac{n}{2}}\det \M_2\neq \square$. 
If $Q^*_2(\m)=0$ then
$c_{p^r}\left(Q^*_2(\m)\right)=\phi(p^r)$. 
It follows from \eqref{eq:even-n} that 
$$ 
\xi_M(s;\m)=L\left(s-1-\frac{n}{2},\psi\right)E_M(s),
$$ 
where $L(s,\psi)$ is the Dirichlet $L$-function associated to the
Jacobi symbol 
$$
\psi(\cdot )=\left(\frac{(-1)^{\frac{n}{2}}\det \M_2}{\cdot}
\right),
$$
with conductor $c_\psi=O(1)$, and where
$E_M(s)$ is an Euler product which converges absolutely
in the half plane $\sigma>\frac{n}{2}+1$ and satisfies the bound
$E_M(s)\ll M^\ve$ there. This gives the analytic continuation
of  $\xi_M(s;\m)$ up to $\sigma>\frac{n}{2}+1$.

Moving the contour of integration in \eqref{perron} to  $c_0=\frac{n+3}{2}$ 
and invoking the convexity estimate \eqref{eq:convex} to deal with the horizontal contours, 
we obtain
$$ 
\sum_{\substack{q\leq x \\(q,M)=1}}\cQ_{q}(\m)=\frac{1}{2\pi
i}\int_{c_0-iT}^{c_0+iT}\xi_M(s;\m)x^s\frac{\d s}{s}+
O\left(\frac{x^c}{T}+\frac{x^{c_0}M^\ve T^\ve}{T^{\frac{3}{4}}}
\right).
$$ 
Here we note that  $(-1)^{\frac{n}{2}}\det \M_2$ is
not a square and so the $L$-series does not have a pole in the region
$\sigma>c_0-\frac{1}{2}$. 
Taking $T=x^{n+4}$ the error term is seen to be 
 $O( x^{-\frac{n}{4}-\frac{3}{2}+\varepsilon}M^\ve)$. 
The remaining integral is estimated via 
\eqref{eq:l-series-bound}, which thereby  leads to the
first part of Lemma \ref{lem:0} when $n$ is even. 

If $n$ is odd and $Q^*_2(\m)=0$, then 
$c_{p^r}\left(Q^*_2(\m)\right)=\phi(p^r)$ and $g_{p^r}\left(Q^*_2(\m)\right)=0$. 
Hence $\xi_M(s;\m)$
is absolutely convergent and bounded by $O(M^\ve)$
in the half-plane $\sigma>\frac{n+3}{2}$.
This implies that we can shift the contour
in \eqref{perron} to $c_0=\frac{n+3}{2}+\ve$, without encountering any poles, 
leading to  a similar but simpler situation to that considered for  even $n$.
This completes the proof of Lemma \ref{lem:0}.
\end{proof}

Let us turn to the size of the exponential sums $\mathcal{Q}_q(\m)$
for generic $\m$, for which sharper bounds are required. 
Tracing through the proof one sees that if $n$ is even and 
$Q_{2}^{*}(\m)\neq 0$ then one is instead led 
to compare $\xi_M(s;\m)$ in \eqref{eq:even-n} with 
$L\left(s-\frac{n}{2},\psi\right)^{-1}$. 
To improve on Lemma \ref{lem:q-triv} 
one therefore requires a good zero-free region for 
$L\left(s-\frac{n}{2},\psi\right)$ to the left of the line $\sigma=\frac{n}{2}+1$, for which the 
unconditional picture is somewhat lacking. 
However, even if one is able to save a power of $x$ in Lemma \ref{lem:q-triv}, this still does not seem to be enough to handle $n=6$ in  Theorem \ref{th1}.
The following result deals with the case of odd $n$ when $Q_2^*(\m)\neq 0$.

\begin{lemma}\label{lem:2}
Let $M \in \NN$ with $N\mid M$ and let $\ve>0$. Assume that $n$ is
odd and $Q_{2}^{*}(\m)\neq 0$. Then we have 
$$ 
\sum_{\substack{q\leq x\\ (q,M)=1}}\cQ_{q}(\m)\ll 
\begin{cases}
|\m|^{\frac{7}{16}+\ve}x^{\frac{n}{2}+1+\ve}M^{\ve}, 
& \mbox{if 
$(-1)^{\frac{n-1}{2}}Q_{2}^{*}(\m)\neq \square$,}\\
x^{\frac{n+3}{2}+\ve}M^{\ve},
& \mbox{if 
$(-1)^{\frac{n-1}{2}}Q_{2}^{*}(\m)= \square$.}
\end{cases}
$$ 
\end{lemma}

\begin{proof}
Recalling \eqref{eq:odd-n} we note that 
$
g_{p}(a)=\chi_{p}(a)\ve(p)p^{\frac{1}{2}},
$
for any non-zero integer $a$ that is coprime to $p$.
Hence we
deduce in this case that
$$ 
\xi_M(s;\m)=L\left(s-\frac{n+1}{2},\psi_{\m}\right)E_M(s),
$$  
where $\psi_{\m}$ is the Jacobi
symbol
$$
\psi_{\m}(\cdot) = 
\left(\frac{(-1)^{\frac{n-1}{2}}Q^*_2(\m)}{\cdot}\right),
$$ 
with conductor $4|Q_2^*(\m)| =O( |\m|^2)$.
Also $E_M(s)$ is an
Euler product which now converges absolutely in the half plane
$\sigma>\frac{n}{2} +1$  
and satisfies the bound  $E_M(s)\ll M^\ve$ there.  
Under the assumption that $(-1)^{\frac{n-1}{2}}Q^*_2(\m)\neq \square$, the $L$-series $\xi_M(s;\m)$ does not have a pole in the region
$\sigma>\frac{n}{2}+1$. 
Moving the contour of integration in
\eqref{perron} to $c_0=\frac{n}{2}+1+\ve$, and using the convexity estimate \eqref{eq:convex}, we therefore get
$$ 
\sum_{\substack{q\leq x\\ 
(q,M)=1}}\cQ_{q}(\m)=\frac{1}{2\pi
i}\int_{c_0-iT}^{c_0+iT}\xi_M(s;\m)x^s\frac{\d s}{s}+
O\left(\frac{x^c}{T}
+\frac{|\m|^{\frac{1}{2}+\ve}x^{c_0}M^\ve}{T^{\frac{3}{4}}}\right),
$$ 
in this case. 
Estimating the remaining integral using
\eqref{eq:l-series-bound}, as before, we conclude the proof of the lemma when
$(-1)^{\frac{n-1}{2}}Q^*_2(\m)\neq \square$  by taking $T$ sufficiently large.

Finally, if 
$(-1)^{\frac{n-1}{2}}Q^*_2(\m)=\square$, then 
$\xi_M(s;\m)$ is regularised by $\zeta(s-\frac{n+1}{2})$ and has a pole at $s=\frac{n+3}{2}$. In this case we move the line of integration back to 
$c_0=\frac{n+3}{2}+\ve$,
which easily leads to the statement of the lemma.
\end{proof}

\section{Analysis of $\cD_d(\m)$}\label{sec:dsum1}

The aim of this section is to collect together everything we need to
know about the sums
$$
\cD_d(\m)=\sum_{\k \in \hat V(\ZZ/d\ZZ)} e_d(\m.\k),
$$
for given $\m \in \ZZ^n$ and $d\in \NN$.  
Here
we write $\hat W$  to denote the affine cone above a projective
variety $W$.  
The estimates in this section pertain to the quadratic forms considered in Theorem~\ref{th1}, so that $V$ is non-singular and we may  
make  use of the geometric facts recorded in \S \ref{geometry}.
Our starting point is Lemma \ref{lem:mult2}, which yields
$
\cD_{d_{1}d_{2}}(\m)=\cD_{d_{1}}(\m)\cD_{d_{2}}(\m)
$
if $(d_{1},d_{2})=1$, rendering it sufficient to understand the
behaviour of the sum at prime powers.

For any $\mathbf{m}\in \ZZ^n$ 
we begin by examining the case in which $d=p$, a prime.  Introducing 
a free sum over elements of $\FF_p^*$, we find that 
\begin{align*}
(p-1)\cD_p(\m)
&=\sum_{a=1}^{p-1}\sum_{\substack{\x\in \hat V(\FF_p)}}e_p(\m.\x)\\
&=\sum_{\substack{\x\in \hat V(\FF_p)}}\sum_{a=1}^{p-1}e_p(a\m.\x)\\
&=p\#\hat V_\m (\FF_p)-\#\hat V (\FF_p),
\end{align*}
where $V_\m$ is the variety obtained by intersecting $V$ with the hyperplane $\m.\x=0$, and 
$\hat V_\m$ is the corresponding affine variety lying above it. 
Rearranging, we obtain
\begin{equation}\label{eq:goat}
\cD_p(\m)=\Big(1-\frac{1}{p}\Big)^{-1}\left(
\#\hat V_\m(\FF_p)-p^{-1}\#\hat V(\FF_p)
\right).
\end{equation}
Now for any  complete intersection $W\subset \PP^m$, which is non-singular modulo $p$ and has dimension $e\geq 1$, it follows from 
Deligne's resolution of the Weil conjectures \cite{deligne} that
$$
|\#W(\FF_p)-(p^{e}+p^{e-1}+\cdots +1)|= O_{d,m}(p^{\frac{e}{2}}),
$$
where $d$ is the degree of $W$. In particular, since 
$$
\#W(\FF_p)=\frac{\#\hat W(\FF_p)-1}{p-1},
$$
we deduce that 
\begin{equation}\label{eq:deligne}
\#\hat W (\FF_p)=p^{e+1} + O_{d,m}(p^{\frac{e+2}{2}}).
\end{equation}
In our setting we have $e=n-3$ for $V$ and $e=n-4$ for $V_\m$ if $p\nmid \m$. 
We may now record the following inequalities.

\begin{lemma}\label{lem:r=1}
We have 
$$
\cD_p(\m)\ll 
\begin{cases}
p^{\frac{n-2}{2}}, & \mbox{if $p\nmid G(\m)$, }\\
p^{\frac{n-1}{2}}, & \mbox{if $p\mid G(\m)$ and $p\nmid \m$, }\\
p^{n-2}, & \mbox{if $p\mid \m$.}
\end{cases}
$$
\end{lemma}

\begin{proof}
 Without loss of generality 
 we may assume
 that $p\nmid
\Delta_{V}$, since otherwise the result is trivial.  
Our starting point is \eqref{eq:goat}. 
If $p\mid \m$ then $\cD_{p}(\m)=\#\hat V(\FF_{p})$ and the claim follows from \eqref{eq:deligne}.

If $p\nmid G(\m)$, so that $V_{\m}$ is non-singular modulo $p$, 
then an application of  \eqref{eq:deligne} yields
\begin{align*}
\cD_p(\m)
&=\Big(1-\frac{1}{p}\Big)^{-1}\left(
p^{n-3}+O(p^{\frac{n-2}{2}})
-p^{-1}(p^{n-2}+O(p^{\frac{n-1}{2}}))\right)=
O(p^{\frac{n-2}{2}}),
\end{align*}
if $n\geq 5$. When $n=4$ this is trivial since then 
$\#V_\m(\FF_p)=O(1)$. 
This establishes the claim. 

Finally, if $p\mid G(\m)$ and $p\nmid \m$, then 
$V_\m$ is singular and of codimension $1$ in $V$ 
modulo $p$. By a result of
Zak (see Theorem 2 in \cite[Appendix]{hooley}), the singular locus of
$V_\m$ has projective dimension $0$. Hence  the work of Hooley 
\cite{hooley} yields
$\# \hat V_\m(\FF_{p})=p^{n-3}+O(p^{\frac{n-1}{2}})$, which once
inserted into  \eqref{eq:goat} yields the desired inequality.
\end{proof}

We now turn our attention to higher prime powers.
Let $d=p^r$ for $r\geq 2$ and suppose that $G(\m)\neq 0$. We assume
that $p\nmid \Delta_V$ and $p\nmid \m$. 
Then it is easy to see that 
$$
\cD_{p^r}(\m)=\sum_{\substack{\x \in \hat V\left(\ZZ\slash p^r\ZZ
    \right)\\p\nmid \x}}e_{p^r}\left(\m.\x\right). 
$$
Mimicking the argument leading to  \eqref{eq:goat},
a line of attack that we already met in the proof of   Lemma \ref{rho(d)},  
we deduce from the explicit formula for the Ramanujan sum that 
$$
\phi(p^r)\mathcal
\cD_{p^r}(\m)=\sideset{}{^*}\sum_{a\bmod{p^r}}\sum_{\substack{\x \in
    \hat V\left(\ZZ\slash p^r\ZZ \right)\\p\nmid
    \x}}e_{p^r}\left(a\m.\x\right)=p^r\sum_{\substack{\x \in \hat
    V\left(\ZZ\slash p^r\ZZ \right)\\p^r|\m.\x\\p\nmid
    \x}}1-p^{r-1}\sum_{\substack{\x \in \hat V\left(\ZZ\slash p^r\ZZ
    \right)\\p^{r-1}|\m.\x\\p\nmid \x}}1. 
$$
In the second sum we write $\x=\y+p^{r-1}\z$ with $\y \bmod{p^{r-1}}$
and $\z \bmod{p}$, to get  
$$
\sum_{\substack{\x \in \hat V\left(\ZZ\slash p^r\ZZ
    \right)\\p^{r-1}|\m.\x\\p\nmid \x}}1=\sum_{\substack{\y \in \hat
    V\left(\ZZ\slash p^{r-1}\ZZ \right)\\p^{r-1}|\m.\y\\p\nmid
    \y}}\#\{\z: \;Q_i(\y+p^{r-1}\z)\equiv 0
\bmod{p^r},\;\;\text{for}\;\;i=1,2\}. 
$$
Since $p\nmid \Delta_V$, the 
count for the number $\z \bmod{p}$ is given by $p^{n-2}$. Setting 
$$
N(p^j,\m)=\#\{\x\in \hat V\left(\ZZ\slash p^j\ZZ \right):\;p\nmid \x,\;\;\m.\x\equiv 0 \bmod{p^j}\},
$$
we get
$$
\mathcal D_{p^r}(\m)=\frac{p^r}{\phi(p^r)}\left\{N(p^r,\m)-p^{n-3}N(p^{r-1},\m)\right\}.
$$
In particular an application of Hensel's lemma yields the following conclusion.

\begin{lemma}\label{lem:r>1}
Let $r\geq 2$. Then we have $\cD_{p^{r}}(\m)=0$ unless $p\mid \Delta_V G(\m)$.
\end{lemma}

We also require a general bound 
for  $\cD_{d}(\m)$. 
By  the orthogonality of characters we may write
$$
\mathcal D_{d}(\ma{m})=\frac{1}{d^{2}}\sum_{\b\bmod{d}} 
\mathcal D_{d}(\ma{m};\b),
$$
where
$$
\mathcal D_{d}(\ma{m};\b)=
\sum_{\k \bmod{d}} e_{d}\left(b_{1}Q_{1}(\k)+b_{2}Q_{2}(\k)+\m.\k\right).
$$
We proceed to extract the greatest common divisor
$h$ of $\b$ with $d$, writing $d=hd'$ and $\b=h\b'$, with
$(d',\b')=1$. Breaking the sum into congruence classes modulo $d'$ we
then see that
$$
\mathcal D_{d}(\ma{m};\b)=\sum_{\k'\bmod{d'}}\sum_{\k'' \bmod{h}} 
e_{d'}\left(b_{1}'Q_{1}(\k')+b_{2}'Q_{2}(\k')+h^{-1}\m.\k'\right)
e_{h}\left(\m.\k''\right).
$$
In particular  $h$ must be a divisor of $\m$ and, furthermore, if we
write $\m=h\m'$ then we have
$
\mathcal D_{d}(\ma{m};\b)=h^n \mathcal D_{d'}(\ma{m}';\b').
$
Applying Lemma \ref{lem:gauss-sum-bound}, we conclude that 
\begin{equation}\label{eq:bh}
|\cD_d(\m)|\leq \frac{1}{d^2} \sum_{h\mid (d,\m)} h^n {d'}^\frac{n}{2}
\sideset{}{^{*}}\sum_{\b' \bmod{d'}}
\sqrt{K_{d'}(2\M(\ma{b}');\0)},
\end{equation}
in the notation of \eqref{eq:2.1} and 
\eqref{eq:Mc}.
The following result provides a good upper bound for the inner sum, provided that $d'$ does not share a common prime factor with $\Delta_V$.

 \begin{lemma}
\label{lem:technical}
For any $\ve>0$ and $e\in \NN$ with $(e,\Delta_V)=1$, we have 
$$
\sideset{}{^{*}}\sum_{\b \bmod{e}}
K_{e}(2\M(\ma{b});\0)\ll e^{2+\ve}.
$$
\end{lemma}

\begin{proof}
Let $g(e)$ denote the sum that is to be estimated
and put $U_{e}(\ma{b})=K_{e}(2\M(\ma{b});\0)$.
One notes via the
Chinese remainder theorem that
$g$ is a multiplicative arithmetic function which it will therefore
suffice to understand at prime powers $e=p^{r}$, with  $p\nmid \Delta_{V}$. 
We have 
$$
g(p^{r})=
\sum_{\substack{
0\leq b_{1},b_{2}<p^{r}\\
p\nmid \b}}
U_{p^{r}}(\ma{b}).
$$
Viewed as a matrix with coefficients in $\ZZ$, it follows from
\eqref{eq:rank} that $\M(\ma{b})$ has rank $n$ or $n-1$, and
furthermore $P(\b)=\det \M(\ma{b})$ has non-zero discriminant, as a polynomial in $\ma{b}$. 
For $i=0,1$ we write $\mathcal{B}_i$ for the set 
of $\ma{b}\in \ZZ^2$ with $0\leq b_1, b_2< p^{r}$ and $p\nmid \b$, for
which $\M(\b)$ has rank $n-i$ over $\ZZ$.

We will provide two upper bounds for $U_{p^r}(\ma{b})$. We begin with Lemma \ref{lem:smith}, which gives
\begin{equation}
  \label{eq:smith}
U_{p^r}(\ma{b})\leq p^{r(n-\rho)+\delta_p},
\end{equation}
where $\rho$ is the rank of $2\M(\ma{b})$ over $\ZZ$ and $\delta_p$ is the
minimum of the $p$-adic orders of the $\rho \times \rho$ non-singular
submatrices of $2\M(\ma{b})$. 
Our second estimate for 
$U_{p^r}(\ma{b})$ is based on an analysis of the case $r=1$. 
Since $p\nmid \Delta_{V}$ it follows that  $2\M(\ma{b})$ has rank $n$
or $n-1$ modulo $p$. In the former case 
one obtains $U_{p}(\ma{b})=1$ and in the latter case
$U_{p}(\ma{b})=p$. An application of Hensel's lemma therefore yields 
\begin{equation}\label{eq:Upr}
U_{p^r}(\b)
\leq 
\begin{cases}
1, & \mbox{if $p\nmid \Delta_{V}\det \M(\ma{b})$,}\\
p^r, & \mbox{if $p\nmid \Delta_{V}$ and $p\mid \det \M(\ma{b})$.}
\end{cases}
\end{equation}
Combining \eqref{eq:smith} and \eqref{eq:Upr} we deduce that
$$
U_{p^r}(\b)
\leq 
\begin{cases}
p^{\min\{r,v_{p}(P(\b))\}}, 
& \mbox{if $\b\in \mathcal{B}_{0}$,}\\
p^r, 
& \mbox{if $\b\in \mathcal{B}_{1}$.}
\end{cases}
$$

It therefore follows that 
$$
g(p^{r})
\leq 
\sum_{\b\in\mathcal{B}_0}
p^{\min\{r, v_p(P(\ma{b}))\}}+
p^{r}\#\mathcal{B}_1.
$$
Now it is clear that there are only $O(1)$ primitive integer solutions
of the equation $P(\b)=0$, whence 
$\#\mathcal{B}_1=O(p^{r})$.  
Moreover we  have 
$
v_p(P(\ma{b}))\leq \Delta$ with $\Delta=rn+O(1)$,
for any $\b\in \mathcal{B}_0$. 
Our investigation so far has shown that for $p\nmid \Delta_{V}$ we have 
$$
g(p^{r})\ll 
 p^{2r}+ \sum_{\ell= 0}^{\Delta} 
p^{\min\{\ell, r\}} \#\mathcal{B}_0(\ell),
$$
where $\mathcal{B}_0(\ell)$ is the set of 
$\b\in \mathcal{B}_0$ for which $p^\ell \mid P(\b)$.
If  $\ell\leq r$ then 
$$
\#\mathcal{B}_0(\ell)\ll p^{2(r-\ell)}
\#\{\b\bmod{p^\ell}: p\nmid \b, ~
P(\b)\equiv 0 \bmod{p^\ell}\}\ll p^{2r-\ell},
$$
since $p$ does not divide the discriminant of $P$.   Alternatively if $\ell>r$
then it follows that 
$$
\#\mathcal{B}_0(\ell)\ll p^{r}.
$$
Putting this altogether we conclude that 
$$
g(p^{r})
\ll p^{2r}+
 \sum_{0\leq \ell\leq r}
p^{2r} 
+ \sum_{r<\ell\leq \Delta} p^{2r}
\ll  r p^{2r},
$$
for $p\nmid \Delta_{V}$.  This suffices for the statement of the lemma.
\end{proof}

Applying Lemma \ref{lem:technical} in \eqref{eq:bh}, we conclude that 
\begin{align*}
\cD_{d}(\m)
&\ll d^{\frac{n}{2}+\ve} (d,\m)^{\frac{n}{2}-2},
\end{align*}
if $(d,\Delta_V)=1$.
If $d\mid \Delta_V^\infty$, we will merely take the trivial bound 
$$
|\cD_d(\m)|\leq \rho(d) \ll d^{n-2+\ve},
$$
which follows from Lemma \ref{rho(d)}.
Combining these therefore leads to the following result.

\begin{lemma}
\label{lem:r rough}
For any $\ve>0$ we have 
$\mathcal D_{d}(\ma{m})\ll 
(d,\Delta_V^\infty)^{\frac{n}{2}-2}
d^{\frac{n}{2}+\ve} (d,\m)^{\frac{n}{2}-2}$. 
\end{lemma}

We are now ready to record some estimates for the average order of $|\cD_{d}(\m)|$,
as we range  over appropriate sets of moduli $d$.
Combining Lemma \ref{lem:r=1}  with Lemma \ref{lem:r>1} and the multiplicativity property in Lemma 
\ref{lem:mult2}, we are immediately led to the following conclusion. 

\begin{lemma}\label{lem:dave}
For any $\ve>0$ we have 
$$
\sum_{\substack{d\leq x\\ (d,\Delta_V G(\m))=1}} |\cD_d(\m)| \ll x^{\frac{n}{2}+\ve}.
$$
\end{lemma}

Here
Lemma \ref{lem:r>1} ensures that only square-free values of $d$ are
counted in this sum. Furthermore this result is trivial if $G(\m)=0$,
in which case we will need an allied estimate. 
This is provided by the following result. 

\begin{lemma}\label{lem:dave'}
Assume that $G(\m)=0$. 
For any $\ve>0$ we have 
$$
\sum_{\substack{d\leq x\\ (d,\Delta_V\m  )=1}} |\cD_d(\m)| \ll 
x^{\frac{n+1}{2}+\ve}.
$$
\end{lemma}

\begin{proof}
We make the factorisation $d=uv$, where $u$ is the square-free part of $d$ and $v$ is
the square-full part.  In particular both $u$ and $v$ are assumed to be coprime to $\Delta_{V}$ and $\m$.
Then Lemma \ref{lem:r=1}   yields
$
\cD_u(\m) \ll u^{\frac{n-1}{2}+\ve},
$
and it follows from 
Lemma \ref{lem:r rough} that
$
\cD_v(\m) \ll v^{\frac{n}{2}+\ve}.
$
Hence 
\begin{align*}
\sum_{\substack{d\leq x\\ (d,\Delta_V\m  )=1}} |\cD_d(\m)| 
&\ll 
\sum_{\substack{uv\leq x}} u^{\frac{n-1}{2}+\ve}
 v^{\frac{n}{2}+\ve} \\
&\ll 
x^{\frac{n+1}{2}+\ve}
\sum_{\substack{v\leq x}} \frac{1}{v^{\frac{1}{2}}}.
\end{align*}
On noting that the number of square-full integers $v\leq V$ is 
$O(V^{\frac{1}{2}})$, this therefore concludes the proof of the lemma.
\end{proof}

\section{Analysis of $\cM_{d,q}(\m)$} 
\label{mcs}

It remains to estimate the mixed character sums $\cM_{d,q}(\m)$,
which it will suffice to analyse at prime powers.  
Our goal in this section will be a proof of the following result.

\begin{lemma}
\label{lem:mixed-strong'}
Assume that $q\mid d^{\infty}$ and $d\mid q^{\infty}$. Let $\ve>0$ and 
assume 
Hypothesis-$\rho$.
Then  we have 
$$
\cM_{d,q}(\m)\ll 
(d,(2\det \M_2)^\infty)^{\frac{n}{2}-2}
d^{\frac{n}{2}+\ve}q^{\frac{n}{2}+1}. 
$$
\end{lemma}

Our proof of this result is based on an analysis of the sum  
\begin{align*}
\mathcal M_{p^r,p^{\ell}}(\m)=\sideset{}{^{*}}\sum_{a \bmod{p^{\ell}}}
\sum_{\substack{\mathbf k \bmod{p^{r+\ell}}\\
Q_1(\k)\equiv 0\bmod{p^r}\\Q_2(\k)\equiv 0\bmod{p^r}}}
e_{p^{r+\ell}}\left(aQ_2(\k)+\m.\k\right),
\end{align*}
for integers $r, \ell\geq 1$.
We first split the inner sum by replacing $\k$ by $\k+p^r\x$, where $\k$ runs modulo $p^r$ and $\x$ runs modulo $p^{\ell}$. This yields
\begin{align}
\label{eq:mixed-step1}
\cM_{p^r,p^{\ell}}(\m)=
\sum_{\substack{\mathbf k \bmod{p^{r}}\\
Q_1(\k)\equiv 0\bmod{p^r}\\Q_2(\k)\equiv 0\bmod{p^r}}}
S(\k),
\end{align}
where
$$
S(\k)=
\sideset{}{^{*}}\sum_{a \bmod{p^{\ell}}}
e_{p^{r+\ell}}\left(aQ_2(\k)+\m.\k\right)
\sum_{\x \bmod{p^{\ell}}}
e_{p^{\ell}}\left(aQ_2(\x)p^r+a\nabla Q_2(\k).\x+\m.\x\right).
$$
We will argue differently according to which of $r$ or $\ell$ is largest. 
Recall that $Q^*_2$ is the dual of $Q_2$, with matrix $\M_2^*=(\det \M_2)\M_2^{-1}$. 
Lemma~\ref{lem:mixed-strong'} is a straightforward consequence of the following pair of results and the multiplicativity property in Lemma \ref{lem:mult2}.

\begin{lemma}
\label{lem:mixed-strong-1}
Suppose that $\ell>r$. Then 
$\cM_{p^r,p^{\ell}}(\m)=0$ unless $p^{r}\mid Q^*_2(\m)$
or $p\mid 2\det \M_2$, in which case 
$\cM_{p^r,p^{\ell}}(\m)\ll p^{\ell+\frac{n}{2}(\ell+r)}$.
\end{lemma}

\begin{proof}
In the inner sum of $S(\k)$ 
we take $\x=\y+p^{\ell-r}\z$, where $\y$ runs modulo $p^{\ell-r}$ and
$\z$ runs modulo $p^r$. This gives 
$$
\sum_{\y \bmod{p^{\ell-r}}}
e_{p^{\ell}}\left(aQ_2(\y)p^r+a\nabla Q_2(\k).\y+ \m.\y\right)\sum_{\z \bmod{p^r}}
e_{p^{r}}\left(a\nabla Q_2(\k).\z+\m.\z\right),
$$
for the sum over
$\x \bmod{p^{\ell}}$.
The sum over $\z$ vanishes unless 
\begin{equation}\label{m:7}
a\nabla Q_2(\k)+\m \equiv \0 \bmod{p^r}.
\end{equation}
Recall from the conditions of summation in \eqref{eq:mixed-step1}
that $p^r\mid Q_2(\k)$.
In particular, if $p\nmid 2\det \M_2$, 
then it follows that 
$\cM_{p^r,p^{\ell}}(\m)=0$ unless $p^{r}\mid Q^*_2(\m)$, as required for the first part of the lemma. For the second part, we let 
$\v\in \ZZ^n$ be such that 
$a\nabla Q_2(\k)+\m =p^r\v$. 
Then we have 
$$
S(\k)=p^{nr} 
\sum_{a \in A(\k)}
e_{p^{r+\ell}}\left(aQ_2(\k)+\m.\k\right)
\sum_{\y \bmod{p^{\ell-r}}}
e_{p^{\ell-r}}\left(aQ_2(\y)+\v.\y\right),
$$
where 
$A(\k)$ denotes the set of $a\in(\ZZ/p^\ell\ZZ)^*$ such that \eqref{m:7} holds.
Applying Lemma \ref{lem:gauss-sum-bound} and then Lemma \ref{lem:smith} 
we conclude that
$$
|S(\k)|\leq\sum_{a \in A(\k)} p^{nr+\frac{n}{2}(\ell-r)} \sqrt{K_{p^{\ell-r}}(2\M_2;\0)}\ll
\sum_{a \in A(\k)}
 p^{nr+\frac{n}{2}(\ell-r)}.
$$
Inserting this into \eqref{eq:mixed-step1} therefore gives
$$
\cM_{p^r,p^{\ell}}(\m)
\ll
p^{\frac{n}{2}(\ell+r)}\sideset{}{^{*}}\sum_{a \bmod{p^{\ell}}}
K_{p^r}(2a\M_2; -\m).
$$
A further application of Lemma \ref{lem:smith} therefore gives
the bound in the lemma.
\end{proof}

\begin{lemma}
\label{lem:mixed-strong-2}
Suppose that $\ell\leq r$ and assume 
Hypothesis-$\rho$.
Then 
$\cM_{p^r,p^{\ell}}(\m)=0$ unless $p^{\ell}\mid Q^*_2(\m)$
or $p\mid 2\det \M_2$, in which case 
$\cM_{p^r,p^{\ell}}(\m)\ll p^{\ell+\frac{n}{2}(\ell+r)} (p,2\det \M_2)^{\frac{nr }{2}-2+\ve}$.
\end{lemma}

\begin{proof}
The expression in \eqref{eq:mixed-step1} now features
$$
S(\k)=
\sideset{}{^{*}}\sum_{a\bmod{p^{\ell}}}e_{p^{r+\ell}}\left(aQ_2(\k)+\m.\k\right)\sum_{\x \bmod{p^{\ell}}}
e_{p^{\ell}}\left(a\nabla Q_2(\k).\x+\m.\x\right).
$$
The sum over $\x$ vanishes unless 
\begin{equation}\label{m:7'}
a\nabla Q_2(\k)+\m \equiv \0 \bmod{p^\ell}.
\end{equation}
Recall that $p^r\mid Q_2(\k)$ in \eqref{eq:mixed-step1}, which implies that
$p^\ell\mid Q_2(\k)$  since $r\geq \ell$. 
If  $p\nmid 2\det \M_2$, it follows from \eqref{m:7'}  that  
$$
a\k\equiv -\overline{2\det \M_2}\M_2^{*}\m \bmod{p^\ell},
$$
whence 
$p^{\ell}\mid Q_{2}^{*}(\m)$, as required for the first part of the lemma.
For the second part we deduce that 
$$
S(\k)=p^{n\ell}
\sideset{}{^{*}}\sum_{\substack{a\bmod{p^{\ell}}\\
\scriptsize{\mbox{\eqref{m:7'} holds}}}}
e_{p^{r+\ell}}\left(aQ_2(\k)+\m.\k\right).
$$
Re-introducing the sum over $\k$ and using exponential sums to
detect the divisibility constraints $p^{r-\ell}\mid p^{-\ell}Q_i(a\k)$, 
which are clearly equivalent to $p^{r-\ell}\mid p^{-\ell}Q_i(\k)$
when $a$ is coprime to $p$, 
we deduce that
\begin{equation}\label{m:8}
\cM_{p^r,p^\ell}(\m)=
\frac{p^{n\ell}}{p^{2(r-\ell)}}
\sum_{\b \bmod{p^{r-\ell}}}
T(\b),
\end{equation}
where
$$
T(\b)=
\sideset{}{^{*}}\sum_{\substack{a\bmod{p^{\ell}}}}
\sum_{\substack{
\k \in K}}
e_{p^{r+\ell}}\left(aQ_2(\k)+\m.\k\right)
e_{p^{r}}\left(b_1Q_1(a\k)+b_2Q_2(a\k)\right),
$$
and 
 $K$ denotes the set of $\k \bmod{p^r}$ for which 
\eqref{m:7'} holds and 
$Q_i(\k)\equiv 0
\bmod{p^{\ell}}$, for $i=1,2$.

We proceed by writing $a\k=\x+p^\ell \y$, for $\y$ modulo $p^{r-\ell}$. 
Let $\bar{a}$ denote the multiplicative inverse of $a$ modulo $p^\ell$, which lifts to a unique point modulo $p^{r+\ell}$.
This leads to the expression
$$
T(\b)=
\sideset{}{^{*}}\sum_{\substack{a\bmod{p^{\ell}}}}
\sum_{\substack{
\x \bmod{p^\ell} \\
\bar{a}\nabla Q_2(\x)+\m \equiv \0 \bmod{p^\ell}\\
Q_i(\x)\equiv 0\bmod{p^\ell}
}}
\sum_{\substack{
\y \bmod{p^{r-\ell}}}}
f(\x,\y),
$$
for $i=1,2$, 
with 
\begin{align*}
f(\x,\y)
&=
e_{p^{r+\ell}}\left(\bar{a}Q_2(\x+p^\ell \y)+\m.(\x+p^\ell \y)\right)
e_{p^{r}}\left(b_1Q_1(\x+p^\ell \y)+b_2Q_2(\x+p^\ell \y)\right).
\end{align*}
Recall the notation $\M(\b)$ introduced in \eqref{eq:Mc}.
One concludes that 
$$
\left|
\sum_{\substack{
\y \bmod{p^{r-\ell}}}}
f(\x,\y)\right|\leq 
\left|\sum_{\substack{
\y \bmod{p^{r-\ell}}}}
e_{p^{r-\ell}}\left(Q(\y)+\ma{n}.\y\right)
\right|,
$$
with 
$\ma{n}=p^{-\ell}(\bar{a}\nabla Q_2(\x)+\m)+2\M(\b)\x$ and 
$$
Q(\y)=\bar{a}Q_2(\y)+p^\ell \left(b_1Q_1(\y)+b_2Q_2(\y)\right).
$$
This quadratic form has underlying matrix $\M(p^\ell b_1,p^\ell b_2+\bar{a})$. The number of $\x \bmod{p^\ell}$ appearing in our expression for $T(\b)$ is $O(1)$ by Lemma \ref{lem:smith}.
Applying Lemma  \ref{lem:gauss-sum-bound}, we deduce that 
$$
T(\b)\ll p^{\frac{(r-\ell)n}{2}}
\sideset{}{^{*}}\sum_{\substack{a\bmod{p^{\ell}}}}
\sqrt{K_{p^{r-\ell}}(2\M
(p^\ell b_1,p^\ell b_2+\bar{a}); \0)}.
$$
As $b_2$ runs modulo $p^{r-\ell}$ and $a$ runs over elements modulo $p^\ell$ which are coprime to $p$, so $c_2=p^\ell b_2+\bar{a}$ runs over a complete set of residue classes modulo $p^r$.  Replacing $b_1$ by $b_1c_2$, and recalling 
 \eqref{m:8}, we obtain
\begin{align*}
\cM_{p^r,p^\ell}(\m)
&\ll 
\frac{p^{\frac{n}{2}(\ell+r)}}{p^{2(r-\ell)}}
\sum_{b_1 \bmod{p^{r-\ell}}}
~\sideset{}{^{*}}\sum_{\substack{c_2\bmod{p^{r}}}}
\sqrt{K_{p^{r-\ell}}(2\M
(p^\ell b_1c_2,c_2); \0)}\\
&\ll 
\frac{p^{\ell+\frac{n}{2}(\ell+r)}}{p^{r-\ell}}
\sum_{b_1 \bmod{p^{r-\ell}}}
\sqrt{K_{p^{r-\ell}}(2\M
(p^\ell b_1,1); \0)}.
\end{align*}
It will be convenient to put $\delta=v_p(2^n \det \M_2)$.
We may assume that $\ell>\delta$. Indeed, if $\ell\leq \delta$ then we
may take the trivial bound $S(\k)=O(1)$ in 
\eqref{eq:mixed-step1}. Applying Hypothesis-$\rho$
we go on to deduce that 
$\cM_{p^r,p^{\ell}}(\m)=O(p^{r(n-2)+\ve})$, which is satisfactory.

Using Taylor's formula we may write 
\begin{align*}
\det 2\M
(p^\ell b_1,1)
&=p^\ell f(b_1)+\det 2\M(0,1)\\
&=
p^\ell f(b_1)+2^n\det \M_2,
\end{align*}
for an appropriate polynomial  $f(b_1)$ with integer coefficients. 
Viewing $b_1$ as an element of $\ZZ$,
it follows that $p^\ell f(b_1)+2^n\det \M_2\neq 0$, since $\ell>\delta$.
Hence 
$$
v_p\left(\det 2\M(p^\ell b_1,1)\right)= \delta
$$
and  Lemma \ref{lem:smith}  yields
$K_{p^{r-\ell}}(2\M
(p^\ell b_1,1); \0)\ll 1$. The overall contribution to 
$\cM_{p^r,p^\ell}(\m)$ from this case is therefore $O(p^{\ell+\frac{n}{2}(\ell+r)})$, which is satisfactory. 
\end{proof}

\section{Proof of Theorem \ref{th1}: initial steps}
\label{pt1}

We henceforth assume that $n\geq 5$.
From Lemma~\ref{psum} we have
$$ 
S_{T,\a}^\sharp(B)=\left(1
+O_N(B^{-N})\right)
\frac{B^{n-2}}{4^n}\sum_{\m\in \mathbb
  Z^n}
    \sum_{\substack{d=1\\ (d,\Delta_V^\infty)\leq \Xi}}^\infty \frac{\chi(d)}{d^{n-1}}
  \sum_{q=1}^\infty\frac{1}{q^n} 
T_{d,q}(\m)I_{d,q}(\m),
$$
for any $N>0$.
We expect that the main term of the sum comes from the zero
frequency $\m=\textbf 0$. This we will compute explicitly in \S \ref{s:conclusion}
and it will turn out to have size $B^{n-2}$, as
expected. Our immediate task, however, is to produce a satisfactory
upper bound for the contribution from the non-zero
frequencies.  In view of the properties of $I_{d,q}(\m)$ recorded in \S \ref{prelim} 
 the sums over $d$ and
$q$ are effectively restricted to $d\ll B$ and
$q\ll Q$, respectively. 
Moreover, Lemma \ref{ubI_q}
implies that the contribution of
the tail $|\m|>dQB^{-1+\varepsilon}$ is arbitrarily small. 
Finally, Lemma \ref{rho(d)} confirms Hypothesis-$\rho$ for the quadratic forms considered here.

As reflected in the various estimates collected together in \S\S
\ref{sec:qsum}--\ref{mcs}, the behaviour of the exponential sum
$T_{d,q}(\m)$ will depend intimately on $\m$. We must
therefore give some thought to the question of controlling the number
of $\m\in \ZZ^n$ which are constrained in
appropriate ways.  
The constraints that feature in our work are of three basic sorts:
either $Q_2^*(\m)=0$ or $G(\m)=0$ or 
$(-1)^{\frac{n-1}{2}}Q_2^*(\m)=\square$, the latter case only being
distinct from the first case when $n$ is odd. 
The first two cases correspond to 
averaging $\m$  over rational points $[\m]$ belonging to a
projective variety $W\subset \PP^{n-1}$, with $W$ equal to the quadric 
$Q_2^*=0$ or the dual hypersurface $V^*$, respectively. 
For such $W$ we claim that 
\begin{equation}
  \label{eq:count}
\#\left\{ \m \in \ZZ^n:~ [\m]\in W(\QQ), ~|\m|\leq M\right\}
\ll M^{n-2+\ve},
\end{equation}
for any $M\geq 1$ and $\ve>0$. 
When $W$ is the quadric, in which case we recall that $Q_2^*$ is
non-singular, this follows from Lemma \ref{m:5}.
When $W=V^*$ then  our discussion in \S \ref{geometry} shows that 
$W$ is an irreducible hypersurface of degree
$4(n-2)\geq 12$. Hence the desired bound follows directly from joint work of the first
author with Heath-Brown and Salberger \cite[Corollary 2]{bhbs}.
Finally, we note that 
\begin{equation}
  \label{eq:count'}
\#\left\{ \m \in \ZZ^n:~ (-1)^{\frac{n-1}{2}}Q_2^*(\m)=\square, ~|\m|\leq M\right\}
\ll M^{n-1+\ve},
\end{equation}
for any $M\geq 1$ and $\ve>0$.  Indeed, the contribution from $\m$ for
which $Q_2^*(\m)=0$ is satisfactory by \eqref{eq:count} and the
remaining contribution leads us to count points of height $O(M)$ on 
a non-singular quadric in $n+1$ variables, for which we may appeal to 
Lemma \ref{m:5}.

We may now return to the task of 
estimating the contribution to 
$ 
S_{T,\a}^\sharp(B)$ 
from $\m$ for which 
$0<|\m|\leq dQB^{-1+\varepsilon}=\sqrt{d}B^\ve$.  In this endeavour it 
 will suffice to study the expression
\begin{equation} 
   \label{eq:AAsum}
U_{T,\a}(B,D)=B^{n-2}\sum_{0<|\m|\leq\sqrt{D}B^{\varepsilon}}
\sideset{}{'}\sum_{\substack{d\sim D\\
(d,\Delta_V^\infty)\leq \Xi}}
\frac{1}{d^{n-1}}\left|\sum_{q}\frac{1}{q^n}
T_{d,q}(\m)I_{d,q}(\m)\right|,
\end{equation}
for $D\geq 1$,  
where $\sum'$ indicates that the sum should be taken over odd integers only
and  the notation $d\sim D$  means $D/2<d\leq D$.
In our analysis of this sum we will clearly only be interested in
values of $D\ll B$.  However, for the time being we allow $D\geq 1$ to be
an arbitrary parameter.

Recall the definition \eqref{eq:NM} of the non-zero integer $N$.
We split $q$ as 
$\delta q$ with $(q,dN)=1$ and $\delta\mid (dN)^{\infty}$.
Since $q$ is restricted to have size $O(Q)$ in \eqref{eq:AAsum}, by the properties of $I_{d,q}(\m)$ recorded in \S \ref{prelim}, we may assume that $\delta\ll B$.
We deduce
from 
the multiplicativity relations Lemma \ref{lem:mult1} and Lemma \ref{lem:mult2} 
that
$$
U_{T,\a}(B,D)
\leq B^{n-2}
\hspace{-0.4cm}
\sum_{0<|\m|\leq\sqrt{D}B^{\varepsilon}}
\sideset{}{'}\sum_{\substack{d\sim D\\
(d,\Delta_V^\infty)\leq \Xi}}
\frac{1}{d^{n-1}}\sum_{\substack{\delta\mid (dN)^{\infty}
\\\delta\ll B}}
\frac{|T_{d,\delta}(\m)|}{\delta^n}\left|
\sum_{\substack{q 
    \\ (q,dN)=1}}\frac{1}{q^n}
\cQ_{q}(\m)I_{d,\delta q}(\m)\right|. 
$$
To estimate the inner sum over $q$ 
we see via partial summation that it is 
\begin{align*}
-\int_{1}^{\infty}\left(\sum_{\substack{q\leq y\\ (q,dN)=1}}
\cQ_{q}(\m)\right)
\frac{\partial}{\partial y}\left(\frac{I_{d,\delta
y}(\m)}{y^n}\right)\d y. 
\end{align*}
The integral is over $y \leq c Q/\delta$, for some absolute constant $c>0$.
Define the quantities
$$
\theta_{1}(n;\m)=\begin{cases}
\frac{7}{16}, & \mbox{if $2\nmid n$ and $(-1)^{\frac{n-1}{2}}Q_{2}^{*}(\m)\neq \square$,}\\
0, &\mbox{otherwise,}
\end{cases}
$$
and 
$$
\theta_2(n;\m)=\begin{cases}
1, &\mbox{
if  $Q_{2}^{*}(\m)=0$ and $(-1)^{\frac{n}{2}}\det\M_{2}=\square$,}\\
\frac{1}{2}, &\mbox{
if  $Q_{2}^{*}(\m)=0$ and $(-1)^{\frac{n}{2}}\det \M_{2}\neq\square$,}\\
\frac{1}{2}, &\mbox{
 if  $Q_{2}^{*}(\m)\neq 0$ and 
$(-1)^{\frac{n-1}{2}}Q_{2}^{*}(\m)=\square$,}\\
0, & \mbox{otherwise.}
\end{cases}
$$
According to our conventions we note that the first case in the definition of 
$\theta_{2}(n;\m)$ only arises for even $n$ and likewise the third case only arises for odd $n$.
Drawing together 
Lemmas~\ref{lem:q-triv}, \ref{lem:0} and \ref{lem:2}, and 
using Lemma \ref{ubI_q2}, we therefore obtain the estimate
\begin{align*}
&\ll |\m|^{\theta_{1}(n;\m)} (dN)^\ve \int_{1}^{cQ/\delta}
y^{\frac{n}{2}+1+\theta_{2}(n;\m)+\ve}
\left|
\frac{\partial}{\partial y}
\left(\frac{I_{d,\delta
y}(\m)}{y^n}\right)\right|\d y\\
&\ll    
\left(\frac{d\delta}{B|\m|}\right)^{\frac{n}{2}-1}
|\m|^{\theta_{1}(n;\m)} (dNB)^\ve \int_{1}^{cQ/\delta}
y^{\frac{n}{2}+1+\theta_{2}(n;\m)}
\cdot 
y^{-\frac{n}{2}-2}
\d y,
\end{align*}
for the above integral.
Let 
\begin{equation}
  \label{eq:def-theta}
  \theta_{1}(n)=\begin{cases}
0, &\mbox{if $n$ is even},\\
\frac{7}{16}, & \mbox{if $n$ is odd},
\end{cases}\quad
  \theta_{2}(n)=\begin{cases}
\frac{1}{2}, & \mbox{if $2\mid n$ and $(-1)^{\frac{n}{2}}\det\M_{2}=\square$,}\\
0, &\mbox{otherwise.} 
\end{cases}
\end{equation}
Returning to our initial estimate for $U_{T,\a}(B,D)
$ and recalling the definition \eqref{eq:S'} of $S_{d,q}(\m)$,
we now  have everything in place to establish
the following result.

\begin{lemma}\label{lem:Asum}
We have 
$$
U_{T,\a}(B,D)
\ll  \frac{B^{\frac{n}{2}-1+\ve}}{D^{\frac{n}{2}}} \left(
U^{(1)}+U^{(2)}\right),
$$
where
$$
U^{(1)}
=
\sum_{\substack{0<|\m|\leq\sqrt{D}B^{\varepsilon}\\ 
(-1)^{\frac{n-1}{2}}Q^*_2(\m)=\square}}
\frac{(B/\sqrt{D})^{\frac{1}{2}+\theta_{2}(n)}}{|\m|^{\frac{n}{2}-1}}
\sum_{\substack{d\sim D\\
(d,\Delta_V^\infty)\leq \Xi}}
\sum_{\substack{\delta\mid d^{\infty}
\\\delta\ll B}}
\frac{|S_{d,\delta}(\m)|}{\delta^{\frac{n}{2}+1}}
$$
and 
$$
U^{(2)}
=
\sum_{\substack{0<|\m|\leq\sqrt{D}B^{\varepsilon}\\ 
(-1)^{\frac{n-1}{2}}Q^*_2(\m)\neq \square
}}
\frac{|\m|^{\theta_{1}(n)}}{|\m|^{\frac{n}{2}-1}}
\sum_{\substack{d\sim D\\
(d,\Delta_V^\infty)\leq \Xi}}
\sum_{\substack{\delta\mid d^{\infty}
\\\delta\ll B}}
\frac{|S_{d,\delta}(\m)|}{\delta^{\frac{n}{2}+1}}.
$$
\end{lemma}

\begin{proof}
Our work so far shows that 
$U_{T,\a}(B,D)\ll C^{(1)}+C^{(2)}$, with 
\begin{align*}
C^{(1)}
&=
\frac{
B^{n-2+\ve}}{B^{\frac{n}{2}-1}}
\sum_{\substack{0<|\m|\leq\sqrt{D}B^{\varepsilon}\\ 
(-1)^{\frac{n-1}{2}}Q_{2}^{*}(\m)=\square
}}
\frac{(B/\sqrt{D})^{\theta_{2}(n;\m)}}{|\m|^{\frac{n}{2}-1}}
\sideset{}{'}
\sum_{\substack{d\sim D\\
(d,\Delta_V^\infty)\leq \Xi}}
\frac{1}{d^{\frac{n}{2}}}
\sum_{\substack{\delta\mid (dN)^{\infty}
\\\delta\ll B}}
\frac{|T_{d,\delta}(\m)|}{\delta^{\frac{n}{2}+1+\theta_{2}(n;\m)}}
\end{align*}
and 
$$
C^{(2)}
=
\frac{B^{n-2+\ve}
}{B^{\frac{n}{2}-1}}
\sum_{\substack{0<|\m|\leq\sqrt{D}B^{\varepsilon}\\ 
(-1)^{\frac{n-1}{2}}Q^*_2(\m)\neq \square}}
\frac{|\m|^{\theta_{1}(n)}}{|\m|^{\frac{n}{2}-1}}
\sideset{}{'}
\sum_{\substack{d\sim D\\
(d,\Delta_V^\infty)\leq \Xi}}
\frac{1}{d^{\frac{n}{2}}}
\sum_{\substack{\delta\mid (dN)^{\infty}
\\\delta\ll B}}
\frac{|T_{d,\delta}(\m)|}{\delta^{\frac{n}{2}+1}}.
$$
We note that $\theta_{2}(n;\m)=\frac{1}{2}+\theta_{2}(n)$ in  $C^{(1)}$,
but take $\frac{n}{2}+1$ for the exponent of $\delta$.
Drawing together Lemma~\ref{lem:mult1}, \eqref{eq:Sell-upper} and 
\eqref{cor:Q_p^r}, it follows
that
$$
\sum_{\substack{\delta\mid N^{\infty}\\\delta\ll B}}
\frac{|T_{1,\delta}(\m)|}{\delta^{\frac{n}{2}+1}}\ll 
\sum_{\substack{\delta\mid N^{\infty}\\\delta\ll B}}1\ll (N
B)^{\varepsilon},
$$ 
where the final inequality follows from  \eqref{eq:scat}. 
Thus 
we can restrict $\delta$ to be a divisor of $d^\infty$ in $C^{(1)}$ and $C^{(2)}$
at the cost of enlarging the bound by $B^{\varepsilon}$.
In particular, since $d$ is odd, it follows that $\delta$ is odd and so Lemma 
\ref{lem:mult1} implies that 
 $T_{d,\delta}(\m)=S_{d,\delta}(\m)$. 
Finally, on taking $d>D/2$ 
in the denominator of both expressions, we
arrive at the statement of the lemma. 
\end{proof}

We are now ready to commence our detailed estimation of $U_{T,\ma{a}}(B,D)$,
based on  Lemma~\ref{lem:Asum}. 
We begin by directing our attention to the estimation of
$U^{(2)}$.
Pulling out the greatest common divisor  $h$ of
$\m$, and then splitting $d=d_1d_2$ and $\delta=\delta_1\delta_2$, with
$\delta_1\mid  d_1^\infty$, $d_1\mid h^{\infty}$, $\delta_2\mid d_2^\infty$ and $(d_2,h)=1$, 
it follows that
\begin{equation} 
\label{eq:A2}
U^{(2)}
=
\sum_{0<h\leq\sqrt{D}B^{\varepsilon}}\frac{h^{\theta_{1}(n)}}{h^{\frac{n}{2}-1}}
\sum_{\substack{0<|\m|\leq\frac{\sqrt{D}B^{\varepsilon}}{h}\\ 
(-1)^{\frac{n-1}{2}}Q^*_2(\m)\neq \square\\ 
\gcd(\m)=1}}
\frac{|\m|^{\theta_{1}(n)}}{|\m|^{\frac{n}{2}-1}}
\sum_{\substack{d_1\leq
D\\d_1\mid h^{\infty}\\
(d_1,\Delta_V^\infty)\leq \Xi
}}\sum_{\substack{\delta_1\mid d_1^{\infty}\\\delta_1\ll B}}
\frac{|S_{d_1,\delta_1}(h\m)|}{\delta_1^{\frac{n}{2}+1}}
\Sigma_1,
\end{equation}
where if $\Xi_{d_1}=\Xi/(d_1,\Delta_V^\infty)$, then 
$$
\Sigma_1=\sum_{\substack{d_2\sim
\frac{D}{d_1}\\(d_2,h)=1\\
(d_2,\Delta_V^\infty)\leq \Xi_{d_1}
}}\sum_{\substack{\delta_2\mid d_2^{\infty}\\\delta_2\ll B}}
\frac{|S_{d_2,\delta_2}(h\m)|}{\delta_2^{\frac{n}{2}+1}}.
$$
Here we recall from \S \ref{prelim} that $S_{d_2,\delta_2}(h\m)=S_{d_2,\delta_2}(\m)$ since $(\delta_2d_2,h)=1$.
Now set 
$$
H(\m)=
\begin{cases}
\Delta_V \det \M_2 G(\m)Q^*_2(\m), & \mbox{if $G(\m)\neq 0$,}\\
\Delta_V \det \M_2 Q^*_2(\m), & \mbox{if $G(\m)=0$,}
\end{cases}
$$
where 
$G$ is the dual form introduced in \S \ref{geometry}. 
Note that $Q_{2}^{*}(\m)\neq 0$ in this definition,
so that $H(\m)$ is a non-zero integer. 

We further split $d_2=d_{21}d_{22}$ and
$\delta_2=\delta_{21}\delta_{22}$ with
$\delta_{21}\mid d_{21}^\infty$, $d_{21}\mid H(\m)^{\infty}$,
$\delta_{22}\mid d_{22}^\infty$ and
$(d_{22},H(\m))=1$. 
It follows that
$$ 
\Sigma_1
\leq 
\sum_{\substack{d_{21}
\leq \frac{D}{d_1}\\d_{21}\mid H(\m)^{\infty}\\
(d_{21},h)=1\\
(d_{21},\Delta_V^\infty)\leq \Xi_{d_1}
}}\sum_{\substack{\delta_{21}\mid d_{21}^{\infty}\\\delta_{21}\ll B}}\sum_{\substack{d_{22}\sim 
\frac{D}{d_1d_{21}}\\(d_{22},hH(\m))=1}}
\sum_{\substack{\delta_{22}\mid d_{22}^{\infty}\\\delta_{22}\ll B}}
\frac{|S_{d_{21},\delta_{21}}(\m)||S_{d_{22},\delta_{22}}(\m)|}{(\delta_{21}\delta_{22})^{\frac{n}{2}+1}}.
$$ 
In view of the fact that 
$(d_{22},2\det \M_2 Q^*_2(\m))=1$, it follows from Lemmas
\ref{lem:mixed-strong-1} and 
\ref{lem:mixed-strong-2}
that $S_{d_{22},\delta_{22}}(\m)$
vanishes unless $\delta_{22}=1$. Hence we may conclude 
that the sum over $d_{22}$ and $\delta_{22}$ is
$$ 
\sum_{\substack{d_{22}\sim
\frac{D}{d_1d_{21}}\\
(d_{22},hH(\m))=1}}|\cD_{d_{22}}(\m)| \ll 
\left(\frac{D}{d_1d_{21}}\right)^{\frac{n}{2}+\psi_1(\m)+\ve},
$$ 
by Lemmas \ref{lem:dave} and \ref{lem:dave'},
where
$$
  \psi_{1}(\m)=
\begin{cases}
\frac{1}{2}, &\mbox{if $G(\m)=0$,}\\
0, &\mbox{otherwise.}
\end{cases}
$$
It follows that
$$ 
\Sigma_1
\ll  \left(\frac{D}{d_1}\right)^{\frac{n}{2}+\psi_{1}(\m)+\ve}
\sum_{\substack{d_{21}
\leq \frac{D}{d_1}\\d_{21}\mid H(\m)^{\infty}\\
(d_{21},h)=1\\
(d_{21},\Delta_V^\infty)\leq \Xi_{d_1}
}}
\sum_{\substack{\delta_{21}\mid d_{21}^{\infty}\\\delta_{21}\ll B}}
\frac{|S_{d_{21},\delta_{21}}(\m)|}{d_{21}^{\frac{n}{2}+
\psi_{1}(\m)}\delta_{21}^{\frac{n}{2}+1}}. 
$$

Now there is a factorisation $d_{21}=d_{21}'d_{21}''$ such that 
$S_{d_{21},\delta_{21}}(\m)=\cM_{d_{21}',\delta_{21}}(\m)\cD_{d_{21}''}(\m)$, 
where $\delta_{21}\mid d_{21}'^{\infty}$. It 
therefore follows from  
Lemma \ref{lem:r rough}  and 
 Lemma \ref{lem:mixed-strong'} that  
$$
S_{d_{21},\delta_{21}}(\m)\ll (d_{21},\Delta_V^\infty)^{\frac{n}{2}-2}
d_{21}^{\frac{n}{2}+\ve}\delta_{21}^{\frac{n}{2}+1},
$$
since $\m$ is primitive. 
Hence 
 \begin{align*}
\Sigma_1
&\ll  \Xi_{d_1}^{\frac{n}{2}-2} \left(\frac{D}{d_1}\right)^{\frac{n}{2}+\psi_{1}(\m)+\ve}
B^{\ve}.
\end{align*}
Substituting this into \eqref{eq:A2} we now examine
\begin{align*} 
\Sigma_{2}
&=
 \sum_{\substack{d_1\leq
D\\d_1\mid h^{\infty}\\
(d_{1},\Delta_V^\infty)\leq \Xi
}}\sum_{\substack{\delta_1\mid d_1^{\infty}\\\delta_1\ll B}}
\frac{|S_{d_1,\delta_1}(h\m)|}{\delta_1^{\frac{n}{2}+1}}
\Sigma_1\\
&\ll 
D^{\frac{n}{2}+\psi_{1}(\m)+\ve}
B^{\ve}
 \sum_{\substack{d_1\leq
D\\d_1\mid h^{\infty}\\
(d_{1},\Delta_V^\infty)\leq \Xi
}}
\frac{\Xi^{\frac{n}{2}-2}
}{(d_{1},\Delta_V^\infty)^{\frac{n}{2}-2}
}
\sum_{\substack{\delta_1\mid d_1^{\infty}\\\delta_1\ll  B}}
\frac{|S_{d_1,\delta_1}(h\m)|}{d_{1}^{\frac{n}{2}+\psi_{1}(\m)}
\delta_1^{\frac{n}{2}+1}}.
\end{align*}
We repeat the process that we undertook above to estimate 
$
S_{d_1,\delta_1}(h\m)$, using 
Lemma \ref{lem:mixed-strong'}  and 
Lemma \ref{lem:r rough}. This gives
$$
\frac{|S_{d_1,\delta_1}(h\m)|}{d_{1}^{\frac{n}{2}+\psi_{1}(\m)}
\delta_1^{\frac{n}{2}+1}}
\ll
(d_{1},\Delta_V^\infty)^{\frac{n}{2}-2}  
d_1^\ve h^{\frac{n}{2}-2-\psi_{1}(\m)}. 
$$
By \eqref{eq:scat}  there are only $O(B^{\ve}D^{\ve})$ values of
$\delta_{1}$ that feature in this analysis.  
In this way we arrive at the estimate
\begin{equation}\label{eq:sigma2}  
\Sigma_{2}
\ll  \Xi^{\frac{n}{2}-2}
D^{\frac{n}{2}+\psi_{1}(\m)+\ve}
B^{\ve}
h^{\frac{n}{2}-2-\psi_{1}(\m)}. 
\end{equation}

It is time to distinguish between whether $G(\m)=0$ or $G(\m)\neq 0$
in our analysis of $U^{(2)}$. Accordingly, let us write
$U^{(2)}=U^{(21)}+U^{(22)}$ for the corresponding decomposition.  
We begin with  a discussion of $U^{(22)}$ , for which $\psi_{1}(\m)=0$
in   \eqref{eq:sigma2} . We deduce from \eqref{eq:A2} that
\begin{align*}
U^{(22)}
&\ll \Xi^{\frac{n}{2}-2}
D^{\frac{n}{2}+\ve}
B^{\ve}
\sum_{\substack{0<h\leq \sqrt{D}B^{\varepsilon}}}
h^{\theta_{1}(n)-1}
\sum_{\substack{0<|\m|\leq\frac{\sqrt{D}B^{\varepsilon}}{h}}}
\frac{|\m|^{\theta_{1}(n)}}{|\m|^{\frac{n}{2}-1}}\\
&\ll \Xi^{\frac{n}{2}-2}
D^{\frac{n}{2}+\ve}
B^{\ve}
\sum_{\substack{0<h\leq \sqrt{D}B^{\varepsilon}}}
h^{\theta_{1}(n)-1}
\left(\frac{\sqrt{D}B^{\varepsilon}}{h}\right)^{\frac{n}{2}+1+\theta_1(n)},
\end{align*}
on breaking the sum over $\m$ into dyadic intervals for $|\m|$. The
sum over $h$ is therefore convergent and we conclude that 
\begin{equation}\label{eq:A22} 
\begin{split}
U^{(22)}
&\ll
\Xi^{\frac{n}{2}-2}D^{\frac{n}{2}+\ve}
B^{\ve}
 \left(\sqrt{D}\right)^{\frac{n}{2}+1+\theta_1(n)}\\
 &=
\Xi^{\frac{n}{2}-2} D^{\frac{3n}{4}+\frac{1+\theta_1(n)}{2}+\ve}
B^{\ve}.
 \end{split}
  \end{equation}

We now turn to a corresponding analysis of $U^{(21)}$, for which 
$\psi_{1}(\m)=\frac{1}{2}$
in   \eqref{eq:sigma2} . It follows from \eqref{eq:A2} that
\begin{align*} 
U^{(21)}
&\ll \Xi^{\frac{n}{2}-2}
D^{\frac{n+1}{2}+\ve}B^{\ve}
\sum_{\substack{0<h\leq \sqrt{D}B^{\varepsilon}}}
h^{\theta_{1}(n)-\frac{3}{2}} 
\sum_{\substack{0<|\m|\leq\frac{\sqrt{D}B^{\varepsilon}}{h}\\ G(\m)=0}}
\frac{|\m|^{\theta_{1}(n)}}{|\m|^{\frac{n}{2}-1}}\\
&\ll \Xi^{\frac{n}{2}-2}
D^{\frac{n+1}{2}+\ve}B^{\ve}
\max_{\frac{1}{2}<M\leq 
\sqrt{D}B^{\varepsilon}}
M^{\theta_{1}(n)+1-\frac{n}{2}}
\sum_{\substack{|\m|\leq M \\ G(\m)=0}}1.
\end{align*}
Appealing to \eqref{eq:count}, we therefore deduce that
\begin{equation}\label{eq:A21} 
\begin{split}
U^{(21)}
&\ll \Xi^{\frac{n}{2}-2}
D^{\frac{n+1}{2}+\ve}
B^{\ve}
\left(\sqrt{D}\right)^{\frac{n}{2}-1+\theta_1(n)}\\
&=
\Xi^{\frac{n}{2}-2}
D^{\frac{3n}{4}+\frac{\theta_1(n)}{2}+\ve}
B^{\ve}.
\end{split}
\end{equation}

Our final task in this section is to estimate $U^{(1)}$ in Lemma
\ref{lem:Asum}, for which we will be able to recycle most of the
treatment of $U^{(2)}$.
Following the steps up to \eqref{eq:sigma2} we find that
$$
U^{(1)}
\ll \Xi^{\frac{n}{2}-2}
D^{\frac{n}{2}+\ve}
\left(\frac{B}{\sqrt{D}}\right)^{\frac{1}{2}+\theta_{2}(n)+\ve}
\sum_{\substack{0<|\m|\leq \sqrt{D}B^{\varepsilon}\\ (-1)^{\frac{n-1}{2}}Q^*_2(\m)=\square}}
D^{\psi_{1}(\m)} 
|\m|^{1-\frac{n}{2}}.
$$
One notes that in the
absence of the function $\theta_{1}(n)$, the exponent of $h$ is at most
$-1$, so that the summation over $h$ can be carried out immediately.
As previously it will be necessary
to write
$U^{(1)}=U^{(11)}+U^{(12)}$, where 
$U^{(11)}$ denotes the contribution from the case 
$G(\m)= 0$ and $U^{(12)}$ is the remaining contribution. 
Beginning with the latter, in which case $\psi_{1}(\m)=0$, we deduce that 
\begin{align*}
U^{(12)}
\ll \Xi^{\frac{n}{2}-2}
D^{\frac{n}{2}+\ve}
\left(\frac{B}{\sqrt{D}}\right)^{\frac{1}{2}+\theta_{2}(n)+\ve}
\max_{\frac{1}{2}<M\leq 
\sqrt{D}B^{\varepsilon}}
M^{1-\frac{n}{2}}
\sum_{\substack{|\m|\leq M \\  (-1)^{\frac{n-1}{2}}Q_2^*(\m)=\square}}1.
\end{align*}
Applying \eqref{eq:count'} we therefore obtain
\begin{equation}\label{eq:A12}
\begin{split}
U^{(12)}
&\ll \Xi^{\frac{n}{2}-2}
D^{\frac{n}{2}+\ve}
\left(\frac{B}{\sqrt{D}}\right)^{\frac{1}{2}+\theta_{2}(n)}
B^\ve 
\left(\sqrt{D}\right)^{\frac{n}{2}}
\\&=\Xi^{\frac{n}{2}-2}
D^{\frac{3n}{4}-\frac{1}{4}-\frac{\theta_2(n)}{2}+\ve}
B^{\frac{1}{2}+\theta_2(n)+\ve}.
\end{split}
\end{equation}
For the remaining contribution, with $\psi_{1}(\m)=\frac{1}{2}$,
we will drop the fact that
$(-1)^{\frac{n-1}{2}}Q_2^*(\m)$ should be a square from the sum over $\m$ since
there is already sufficient gain from the fact that $G(\m)$ 
vanishes. Arguing as above, but this time with recourse to
\eqref{eq:count}, we conclude that
\begin{equation}\label{eq:A11}
\begin{split}
U^{(11)}
&\ll
\Xi^{\frac{n}{2}-2}
D^{\frac{n+1}{2}+\ve}
\left(\frac{B}{\sqrt{D}}\right)^{\frac{1}{2}+\theta_{2}(n)}
B^\ve 
\left(\sqrt{D}\right)^{\frac{n}{2}-1}\\
&=
\Xi^{\frac{n}{2}-2}
D^{\frac{3n}{4}-\frac{1}{4}-\frac{\theta_2(n)}{2}+\ve}
B^{\frac{1}{2}+\theta_2(n)+\ve}.
\end{split}
\end{equation}

Recall the definitions \eqref{eq:def-theta} of $\theta_{1}$ and $\theta_{2}$.
Combining \eqref{eq:A22}--\eqref{eq:A11} in Lemma \ref{lem:Asum}, we may now record our final
bound for 
$U_{T,\a}(B,D)$.

\begin{lemma}\label{lem:Asum'}
Let $n\geq 5$ and $ D\geq 1$. Then we have 
\begin{align*}
U_{T,\a}(B,D) 
\ll \Xi^{\frac{n}{2}-2} B^{\frac{n}{2}-1+\ve} 
\left(
D^{\frac{n}{4}+\frac{1+\theta_1(n)}{2}+\ve}
+
D^{\frac{n}{4}-\frac{1}{4}-\frac{\theta_2(n)}{2}+\ve}
B^{\frac{1}{2}+\theta_2(n)}\right).
\end{align*}
\end{lemma}

 \section{Proof of Theorem \ref{th1}: conclusion}
\label{s:conclusion}

Recall the expression for $S_{T,\a}^\sharp(B)$ recorded at the start of \S \ref{pt1}.
We now have everything in place to estimate the overall contribution
to this sum from the non-zero $\m$. An upper bound for this
contribution is obtained by taking $D\ll B$ in  
Lemma \ref{lem:Asum'}'s estimate  for the
quantity introduced in \eqref{eq:AAsum}. 
This gives the overall contribution 
\begin{equation*}
\begin{split}
&\ll \Xi^{\frac{n}{2}-2}
B^{\frac{n}{2}-1+\ve} \left(
B^{\frac{n}{4}+\frac{1+\theta_{1}(n)}{2}}
+
B^{\frac{n}{4}+\frac{1}{4}+\frac{\theta_{2}(n)}{2}}
\right)\\
&\ll \Xi^{\frac{n}{2}-2}
B^{\frac{3n}{4}-\frac{1}{2}+\frac{\theta_{1}(n)}{2} +\varepsilon}.
\end{split}
\end{equation*}
Combining  this with Lemma \ref{S(B)}, Lemma \ref{lem:flat} 
and Lemma \ref{psum}, our work so far has shown that 
\begin{equation}\label{eq:nose}
S(B)= 
M^\sharp(B)+
O(\Xi^{-\frac{1}{n}}B^{n-2+\ve}+\Xi B^{n-3+\ve}
+
\Xi^{\frac{n}{2}-2} B^{\frac{3n}{4}-\frac{1}{2}+\frac{\theta_{1}(n)}{2} +\varepsilon}),
\end{equation}
where
\begin{equation}\label{eq:main}
M^\sharp (B)
=\frac{B^{n-2}}{4^{n-1}} 
\sum_T
\sum_{\substack{
\a\in (\mathbb Z/4\mathbb Z)^n\\
Q_1(\a)\equiv 1 \bmod{4}}}
\sum_{\substack{d=1\\
(d,\Delta_V^\infty)\leq \Xi}}^\infty \frac{\chi(d)}{d^{n-1}}
  \sum_{q=1}^\infty\frac{1}{q^n} 
T_{d,q}(\ma{0})I_{d,q}(\ma{0}).
\end{equation}

We begin with a few words about the integral
$$
I_{d,q}(\ma{0})=
\int_{\mathbb R^n}
h\left(\frac{q\sqrt{d}}{B}, Q_2(\y)\right)W_d(\y) \d \mathbf y,
$$
where $W_d$ is given by \eqref{eq:W_d}
and 
we have made the substitution $Q=B/\sqrt{d}$. Recall the correspondence
\eqref{eq:I*}
between $I_{d,q}(\ma{0})$ and $I_r^*(\ma{0})$.
Recall additionally the properties of $h(x,y)$ and the weight function $W_d$ that were recorded in \S \ref{prelim}. In particular 
$\nabla Q_2(\y)\gg 1$ on $\supp(W_d)$ and 
we have 
$d\ll B$ and 
 $q\sqrt{d}\ll B$ if $I_{d,q}(\ma{0})$ is non-zero. Combining \cite[Lemma 14]{H} and 
\cite[Lemma~15]{H} it follows that 
\begin{equation}\label{eq:I-trivial}
I_{d,q}(\ma{0})\ll 1.
\end{equation}
Furthermore, according to \cite[Lemma 13]{H}, we have
\begin{equation}\label{eq:home}
I_{d,q}(\ma{0})=
\tau_{\infty}(Q_2,W_d)+O_N\left\{\left( \frac{q\sqrt{d}}{B}\right)^N\right\},
\end{equation}
for any $N>0$, where for any infinitely differentiable bounded 
function  $\omega: \RR^n\rightarrow \RR$ of compact support we set
\begin{equation}\label{eq:tau-inf}
\tau_{\infty}(Q_2,\omega)=
\lim_{\varepsilon\rightarrow 0}(2\varepsilon)^{-1}\int_{|Q_2(\y)|\leq \varepsilon}\omega(\y)\d\y.
\end{equation}
 In fact 
$\tau_{\infty}(Q_2,\omega)$ is the real density of points on the affine cone over
the hypersurface  $Q_2=0$, weighted by $\omega$.
We will use these facts to extract the dependence on $I_{d,q}(\ma{0})$ from \eqref{eq:main}. 

Returning to \eqref{eq:main}, our main goal in this section will be a proof of the following 
asymptotic formula.

\begin{lemma}\label{lem:main}
Let $n\geq 5$, let $\ve>0$ and assume Hypothesis-$\rho$.
Then we have
$$
M^\sharp(B)= 
B^{n-2}
\sigma_\infty \prod_p \sigma_p 
+O(\Xi^{-1}B^{n-2}+
B^{n-\frac{5}{2}+\ve}+ 
B^{\frac{3n}{4}-1+\ve}),
$$
where $\sigma_\infty$ and $\sigma_p$ are the expected local densities of points on $X(\RR)$ and  $X(\QQ_p)$, 
respectively.  In particular 
$
\sigma_\infty \prod_p \sigma_p >0
$
if $X(\RR)$ and $X(\QQ_p)$ are non-empty for each prime $p$.
\end{lemma}

In the context of 
Theorem \ref{th1}, for which $n\geq 7$, we note that
Hypothesis-$\rho$ follows from Lemma \ref{rho(d)}.
We now wish to apply Lemma \ref{lem:main} in  \eqref{eq:nose}
to complete the proof of Theorem~\ref{th1}.
Our estimates will be optimised by the choice  $\Xi=B^{\xi(n)}$, with 
$$
\xi(n)=\left(\frac{n}{4}-\frac{3+\theta_1(n)}{2}\right)\left(\frac{2n}{n^2-4n+2}\right),
$$
which comes from balancing the first and third error terms in \eqref{eq:nose}.
We make the observation that  $
\xi(n)<\xi(n)(1+\frac{1}{n})<1,
$
for $n\geq 7$.  Hence we obtain the overall error term $O(B^{n-2-\eta(n)+\ve})$, with
$$
\eta(n)
=\min\left\{ 
\frac{\xi(n)}{n}, ~
1-\xi(n), ~\frac{1}{2}, ~\frac{n}{4}-1
\right\}
=
\frac{\xi(n)}{n}.
$$
Observe that $\eta(n)>0$ if $n\geq 7$. At this point we stress that if we
had exponent $\frac{1}{2}$ instead of $\frac{7}{16}$ in \eqref{eq:l-series-bound}, which
corresponds to the convexity bound, we would have
$\theta_{1}(n)=\frac{1}{2}$ for odd $n$ and hence our result would only hold for
$n\geq 8$.
This completes the proof of Theorem \ref{th1}, subject to Lemma \ref{lem:main}.

\medskip

The remainder of this section will be devoted to the proof of Lemma \ref{lem:main}.
Combining \eqref{eq:Sell-upper}, \eqref{cor:Q_p^r} and Lemma \ref{lem:mixed-strong'} it follows from Hypothesis-$\rho$
that 
\begin{equation}\label{eq:hyp-est} 
T_{d,q}(\ma{0})\ll d^{n-2+\ve}q^{\frac{n}{2}+1},
\end{equation}
for any $d,q\in \NN$.
We will also make use 
of the bound 
\eqref{eq:I-trivial} and the fact that 
$d\ll B$ whenever $I_{d,q}(\ma{0})$ is non-zero.
Let $M(B)$ be defined as in \eqref{eq:main}, but in which the sum over $d$ runs over all positive integers.  It follows from  \eqref{eq:hyp-est} that 
$M^\sharp(B)=M(B)+O(\Xi^{-1}B^{n-2+\ve})$. 
Write 
$$
M(B)
=\frac{B^{n-2}}{4^{n-1}}\sum_T M_T(B),
$$
say. 

For given $\theta>0$,
let us consider the contribution to $M_T(B)$  from $q>B^{\frac{1}{2}-\theta}$. Invoking  \eqref{eq:hyp-est}, this  contribution is seen to be
\begin{align*}
&\ll
\sum_{\substack{
\a\in (\mathbb Z/4\mathbb Z)^n\\
Q_1(\a)\equiv 1 \bmod{4}}}
\sum_{d\ll B}\frac{1}{d^{n-1}}
  \sum_{q>B^{\frac{1}{2}-\theta}}\frac{|T_{d,q}(\ma{0})|}{q^n} 
\\
&\ll
\sum_{d\ll B}d^{-1+\ve}
  \sum_{q>B^{\frac{1}{2}-\theta}} 
q^{-\frac{n}{2}+1+\ve}\\
&\ll B^{(\frac{1}{2}-\theta)(-\frac{n}{2}+2)+\ve},
\end{align*}
since $n\geq 5$.
Turning to the contribution 
from $q\leq B^{\frac{1}{2}-\theta}$ we see that the error term in 
\eqref{eq:home} is $O_N(B^{-N})$ for arbitrary $N>0$, since $d\ll B$.
Hence such $q$ make the overall contribution
$$
\sum_{\substack{
\a\in (\mathbb Z/4\mathbb Z)^n\\
Q_1(\a)\equiv 1 \bmod{4}}}
\sum_{d=1}^\infty\frac{\chi(d)\tau_\infty(Q_2,W_d)}{d^{n-1}} 
  \sum_{q\leq B^{\frac{1}{2}-\theta}}\frac{T_{d,q}(\ma{0})
}{q^n} 
+O_N(B^{-N}),
$$
to $M_T(B)$. The previous paragraph shows that the summation over $q$ can be extended to infinity with error 
$O(B^{(\frac{1}{2}-\theta)(-\frac{n}{2}+2)+\ve})$. Taking $\theta$ to be a suitably small positive multiple of $\ve$, we may therefore conclude that 
$$
M_T(B)=
\sum_{\substack{
\a\in (\mathbb Z/4\mathbb Z)^n\\
Q_1(\a)\equiv 1 \bmod{4}}}
\sum_{d=1}^\infty\frac{\chi(d)\tau_\infty(Q_2,W_d)}{d^{n-1}}
  \sum_{q=1}^\infty
  \frac{T_{d,q}(\ma{0})
}{q^n} 
+O( B^{-\frac{n}{4}+1+\ve}).
$$
Let us denote by $L_T(B;W_d)$ the main term in this expression.
We proceed to introduce the summation over $T$ via the  following result, in which $\rho(d)=\cD_d(\ma{0})$.

\begin{lemma}\label{lem:d-m0}
Let $\ve>0$ and  $M\in \NN$.   Assume Hypothesis-$\rho$.
Then for any $1\leq y<x$ we have 
$$
\sum_{\substack{y<d\leq x\\ (d,M)=1}} \frac{\chi(d) \rho(d)}{d^{n-1}} \ll \frac{M^\ve}{\sqrt{y}}.
$$
\end{lemma}

\begin{proof}
Let $s=\sigma+it\in \CC$. 
In the usual way we consider the Dirichlet series
$$
\eta_M(s)=\sum_{(d,M)=1} 
\frac{\chi(d) \rho(d)}{d^{s}} 
=\prod_{p\nmid M} \left(1+\frac{\chi(p)\rho(p)}{p^s} +O(p^{2n-4-2\sigma+\ve})
\right),
$$
where the error term comes from Hypothesis-$\rho$.
Since 
$\rho(p)=p^{n-2}+O(p^{n-\frac{5}{2}})$, by the Lang--Weil estimate, we  conclude that 
$$
\eta_M(s)=L\left(s-(n-2),\chi\right) E_M(s),
$$
where $E_M(s)$ is absolutely convergent and bounded by $O(M^\ve)$ for $\sigma>n-\frac{3}{2}$.
The conclusion of the lemma is now available through a straightforward application 
of Perron's formula in the form \eqref{perron}.
\end{proof}

We deduce from  Lemma \ref{lem:d-m0} that
$$
\sum_{\substack{(d,q)=1}} \frac{\chi(d) \rho(d) V_T(d)}{d^{n-1}} 
\ll \frac{ q^\ve }{\sqrt{T}}
$$
and 
$$
\sum_{\substack{(d,q)=1}} \frac{\chi(d) \rho(d) V_T(B^2Q_1(\y)/d)}{d^{n-1}} 
\ll \frac{ q^\ve \sqrt{T}}{B\sqrt{Q_1(\y)}}
\ll \frac{ q^\ve \sqrt{T}}{B},
$$
for any $\y\in \supp(W)$. Here we recall that $Q_1(\y)$ is positive  and has order of magnitude $1$ on $\supp(W)$.

We now claim that 
$$
\sum_{T}L_T(B;W_d)=2C
+O(B^{-\frac{1}{2}+\ve}),
$$
with 
\begin{equation}\label{eq:CC}
C=\tau_\infty(Q_2,W)
\sum_{\substack{
\a\in (\mathbb Z/4\mathbb Z)^n\\
Q_1(\a)\equiv 1 \bmod{4}}}
\sum_{d=1}^\infty\frac{\chi(d)}{d^{n-1}} 
  \sum_{q=1}^\infty
  \frac{T_{d,q}(\ma{0})
}{q^n}.
\end{equation}
Now the weight function $W_d$ differs according to whether $T\leq B$ or $T>B$. 
It will be convenient to set 
$W^{(1)}(\y)=W(\y)V_T(d)$ and 
$W^{(2)}(\y)=W(\y)V_T(B^2Q_1(\y)/d)$.
In either case we wish to extend the sum over $T$ to the full range, since 
$\sum_T V_T(t)=1$ for $1\leq t\ll B^2$. 
We have 
$$
\sum_{T\leq B} L_T(B;W_d)
 = C -\sum_{T>B}L_T(B;W^{(1)}),
$$
and 
$$
\sum_{T> B} L_T(B;W_d)
 = C -\sum_{T\leq B}L_T(B;W^{(2)}).
$$
To estimate the tails we
employ the factorisation properties of 
$T_{d,q}(\ma{0})$, finding that 
$$
L_T(B;W^{(i)})=
\sum_{\substack{
\a\in (\mathbb Z/4\mathbb Z)^n\\
Q_1(\a)\equiv 1 \bmod{4}}}
  \sum_{q=1}^\infty
  \frac{1}{q^n} 
\sum_{\substack{\delta \mid q\\ \delta \ll B}} 
  \frac{T_{\delta,q}(\ma{0})}{\delta^{n-1}} 
\sum_{(d,q)=1}\frac{\chi(d)\rho(d)\tau_\infty(Q_2,W^{(i)})}{d^{n-1}},
$$
for $i=1,2$.
The claim is now an easy consequence of our hypothesised bound  
\eqref{eq:hyp-est} and Lemma \ref{lem:d-m0}.
 Bringing everything together, we have therefore shown that 
\begin{equation}\label{eq:t3}
M(B)
= \frac{2B^{n-2}}{4^{n-1}}
C +O(B^{n-\frac{5}{2}+\ve}+ 
B^{\frac{3n}{4}-1+\ve}),
\end{equation}
with $C$ given by 
\eqref{eq:CC}.

We wish to show that the leading 
constant  admits an interpretation in terms of local densities for the intersection of quadrics $X$ considered in Theorem \ref{th1}.
For a prime $p$ the relevant $p$-adic density is equal to
$$
\sigma_p=\lim_{k\rightarrow \infty} p^{-kn}
N(p^k),
$$
where 
$$
N(p^k)=
\#\left\{
(\x,u,v)\in (\ZZ/p^k\ZZ)^{n+2}:  
\begin{array}{l}
Q_1(\x)\equiv u^2+v^2 \bmod{p^k},\\
Q_2(\x)\equiv 0 \bmod{p^k}
\end{array}
\right\},
$$
if $p>2$, and 
$$
N(2^k)=
\#\left\{
(\x,u,v)\in (\ZZ/2^k\ZZ)^{n+2}:  
\begin{array}{l}
Q_1(\x)\equiv u^2+v^2 \bmod{2^k},\\
Q_2(\x)\equiv 0 \bmod{2^k}, ~2\nmid Q_1(\x)
\end{array}
\right\}.
$$
The restriction to odd values of $Q_1(\x)$ in  $N(2^k)$ comes from the definition of the counting function $S(B)$.
In order to relate these densities to the local factors that arise in our analysis,  we 
set
$$
  S(A;p^k)=\#\{(u, v) \in(\ZZ/p^k\ZZ)^2: u^2+v^2 \equiv A \bmod{p^k}\},
$$
for any  $A\in\ZZ$ and any  prime power $p^k$.
According to Heath-Brown \cite[\S 8]{h-b03} we have 
$$
S(A;p^k)=\begin{cases}
p^k+kp^k(1-1/p), &\mbox{if $v_p(A)\geq k$},\\
(1+v_p(A))p^{k}(1-1/p),
& \mbox{if $v_p(A)<k$},
\end{cases}
$$
when $p\equiv 1\bmod{4}$.  When $p\equiv 3 \bmod{4}$, we
have 
$$
S(A;p^k)=\begin{cases}
p^{2[\frac{k}{2}]}, &\mbox{if $v_p(A)\geq k$},\\
p^{k}(1+1/p), & \mbox{if $v_p(A)<k$ and $2\mid v_p(A)$},\\
0, & \mbox{if $v_p(A)<k$ and $2\nmid v_p(A)$}.
\end{cases}
$$
Finally, for odd $A$,  when $p=2$ and $k\geq 2$  we have
$$
S(A;2^k)=\begin{cases}
2^{k+1}, & \mbox{if $A\equiv 1\bmod{4}$,}\\
0, & \mbox{otherwise.}
\end{cases}$$

We now have everything in place to reinterpret the densities $\sigma_p$. 
We being by analysing the case $p=2$, obtaining 
\begin{equation}\label{eq:sig2}
\sigma_2=
\lim_{k\rightarrow \infty} 2^{1-k(n-1)}\#\left\{
\x\in (\ZZ/2^k\ZZ)^{n}:  
\begin{array}{l}
Q_1(\x)\equiv 1\bmod{4},\\
Q_2(\x)\equiv 0 \bmod{2^k}
\end{array}
\right\}.
\end{equation}
Alternatively, when $p>2$, it is straightforward to deduce that
\begin{equation}\label{eq:sigp}
\sigma_p=
\left(1-\frac{\chi(p)}{p}\right)
\lim_{k\rightarrow \infty} p^{-k(n-1)}
\sum_{0\leq e\leq k}
\chi(p^e)
\widetilde{N}_k(e),
\end{equation}
where 
$$
\widetilde{N}_k(e)=
\#\left\{
\x\in (\ZZ/p^k\ZZ)^{n}:  
\begin{array}{l}
Q_1(\x)\equiv 0\bmod{p^e},\\
Q_2(\x)\equiv 0 \bmod{p^k}
\end{array}
\right\}.
$$
Finally, for  the real density $\sigma_\infty$ of points, 
we claim that 
\begin{equation}\label{eq:siginf}
\sigma_\infty
=\pi  \tau_\infty(Q_2,W),
\end{equation}
in the notation of \eqref{eq:tau-inf}. 
Supposing that the equations
for $X$ are taken to be $Q_1(\x)=u^2+v^2$ and $Q_2(\x)=0$,
the real density is equal to 
$$
\sigma_\infty =\int_{-\infty}^\infty
\int_{-\infty}^\infty \int_{(\x,u,v)\in \RR^{n+2}} 
\hspace{-0.7cm}
W(\x) e\left( \alpha \{Q_1(\x)-u^2-v^2\} +\beta Q_2(\x)\right) \d\x \d u \d v\d \alpha \d \beta.
$$
We restrict $u,v$ to be non-negative and substitute 
$t=Q_1(\x)-u^2-v^2$ for $v$. Writing  
$$
F(t)=
\frac{1}{2}
\int_{-\infty}^\infty 
\int_{\x,u} \frac{W(\x) e\left(\beta Q_2(\x)\right) }{\sqrt{Q_1(\x)-u^2-t}} \d \x \d u \d \beta,
$$
where the integral is over $(\x,u)\in \RR^{n+1}$ such that $u\geq 0$ and 
$Q_1(\x)-u^2-t\geq 0$, we 
therefore obtain
$$
\sigma_\infty =4
\int_{-\infty}^\infty 
\int_t F(t)e(\alpha t) \d t\d \alpha.
$$
By the Fourier inversion theorem this reduces to $4F(0)$.
Noting that 
$$
\int_0^{\sqrt{A}} \frac{\d u}{\sqrt{A-u^2}}=\frac{\pi}{2},
$$
for any $A>0$, we arrive at the expression
$$
\sigma_\infty=4\times
\frac{1}{2}\times \frac{\pi}{2}
\int_{-\infty}^\infty 
\int_{\x\in \RR^n} W(\x) e\left(\beta Q_2(\x)\right)  \d \x  \d \beta.
$$
But the remaining integral is just the real density $\tau_\infty(Q_2,W)$, by \cite[Theorem 3]{H}.  This concludes the proof of
\eqref{eq:siginf}.

It is now time to interpret the constant 
$C$ in \eqref{eq:CC}
in terms of the local densities $\sigma_p$ and $\sigma_\infty$.
Invoking  Lemma \ref{lem:mult1} we may write
$$
C=
\tau_{\infty}(Q_2,W)
\sum_{\substack{
\a\in (\mathbb Z/4\mathbb Z)^n\\
Q_1(\a)\equiv 1 \bmod{4}}}
\sum_{d=1}^\infty\frac{\chi(d)}{d^{n-1}}
\sum_{\ell=0}^\infty
  \sum_{\substack{q'=1\\ 2\nmid q'}}^\infty
  \frac{1}{(2^\ell q')^n} 
S_{d,q'}(\ma{0})
S_{1,2^\ell}^{\chi(dq')}(\ma{0}).
$$
Recall \eqref{eq:Sell}.  We therefore see that for fixed $d$ and $q'$ the sum over $\a$ and $\ell$ is
\begin{align*}
\sum_{\substack{
\a\in (\mathbb Z/4\mathbb Z)^n\\
Q_1(\a)\equiv 1 \bmod{4}}}
\sum_{\ell=0}^\infty
  \frac{1}{2^{\ell n}} 
S_{1,2^\ell}^{\chi(dq')}(\ma{0})
&=
\sum_{\ell=0}^\infty
  \frac{1}{2^{\ell n}}
  \sideset{}{^{*}}\sum_{a\bmod{2^\ell}}
\sum_{
\substack{
\ma{k}\bmod{2^{2+\ell}}\\
Q_1(\ma{k})\equiv 1 \bmod{4}}}
 e_{2^{\ell}}
\left(a Q_2(\ma{k})\right)\\
&=
\lim_{\ell \rightarrow\infty} 2^{-\ell(n-1)} 
\#\left\{
\k\in (\ZZ/2^{2+\ell}\ZZ)^{n}:  
\begin{array}{l}
Q_1(\k)\equiv 1\bmod{4},\\
Q_2(\k)\equiv 0 \bmod{2^\ell}
\end{array}
\right\}\\
&=4^{n}\times  \frac{\sigma_2}{2},
\end{align*}
on carrying out the sum over $a$ and comparing with \eqref{eq:sig2}.
Hence it follows that 
$$
C=
4^{n}\times  \frac{\sigma_2}{2}\times
\tau_{\infty}(Q_2,W)
\sum_{d=1}^\infty\frac{\chi(d)}{d^{n-1}}
  \sum_{\substack{q'=1\\ 2\nmid q'}}^\infty
  \frac{1}{q'^n} 
S_{d,q'}(\ma{0}).
$$
Expressing the sum over $d$ and $q'$ as an Euler product one finds that 
\begin{align*}
\sum_{d=1}^\infty\frac{\chi(d)}{d^{n-1}}
  \sum_{\substack{q'=1\\ 2\nmid q'}}^\infty
  \frac{1}{q'^n} 
S_{d,q'}(\ma{0})
&=
\prod_{p>2}
\sum_{r,\ell\geq 0}
\frac{p^r\chi(p^r)}{p^{(r+\ell)n}}
 S_{p^r,p^\ell}(\ma{0}).
\end{align*}
Here $S_{p^r,1}(\ma{0})=\widetilde{N}_r(r)$ and 
$
S_{p^r,p^\ell}(\ma{0})=p^\ell \widetilde{N}_{r+\ell}(r) - 
p^{\ell-1+n} \widetilde{N}_{r+\ell-1}(r),
$
when $\ell\geq 1$, in the notation of \eqref{eq:sigp}.  It easily follows that
$$
\sum_{d=1}^\infty\frac{\chi(d)}{d^{n-1}}
  \sum_{\substack{q'=1\\ 2\nmid q'}}^\infty
  \frac{1}{q'^n} 
S_{d,q'}(\ma{0})
=\prod_{p>2} 
\tau_p,
$$
with
\begin{align*}
\tau_p&=
\lim_{k\rightarrow \infty} p^{-k(n-1)}\sum_{0\leq r\leq k} \chi(p^r)\widetilde{N}_k(r) 
=\left(1-\frac{\chi(p)}{p}\right)^{-1} \sigma_p.
\end{align*}
Finally, on appealing to the identity \eqref{eq:siginf} and noting that $L(1,\chi)=\pi/4$, we deduce that
\begin{align*}
C
&=
4^{n}\times  \frac{\sigma_2}{2} \times
\tau_{\infty}(Q_2,W)
L(1,\chi)\prod_{p>2} 
\sigma_p\\
&=
\frac{4^{n-1}}{2}\times 
\sigma_{\infty}
\prod_{p} 
\sigma_p.
\end{align*}
Once inserted into \eqref{eq:t3} we 
therefore arrive at the statement of Lemma \ref{lem:main}.

\end{document}